\documentclass[twoside]{amsart}
\usepackage{graphicx}
\usepackage{amsthm}
\usepackage{amsmath}
\usepackage{amssymb}
\usepackage{bbm}

\usepackage{color}
\usepackage{hyperref}
\usepackage{bbm}
\usepackage[T1]{fontenc}

\newtheorem{thm}{Theorem}[section]

\newtheorem{cor}[thm]{Corollary}

\newtheorem{prop}[thm]{Proposition}
\theoremstyle{definition}
\newtheorem{defn}[thm]{Definition}
\newtheorem{example}{Example}

\theoremstyle{remark}
\newtheorem{rem}[thm]{Remark}

 \numberwithin{equation}{section}
\newcommand{\norm}[1]{\left\Vert#1\right\Vert}
\newcommand{\abs}[1]{\left\vert#1\right\vert}
\newcommand{\set}[1]{{\left\{#1\right\}}}

\newcommand{\I}{\mathbbm{1}}
\newcommand{\Tr}{ \text{Tr} }
\newcommand{\mtop}{{\!\top}}
\newcommand{\cadlag}{c\`{a}dl\`{a}g }

\title[Jump-diffusions ]{Jump-diffusion processes in random environments}
\author{Jacek Jakubowski}\thanks{Research partially supported by Polish MNiSW grant N N201 547838.}
\address{
Institute of Mathematics, University
of Warsaw, \\
 ul. Banacha 2, 02-097 Warszawa, Poland \\
 and \\
 Faculty of
Mathematics and
Information Science,  \\
Warsaw University of
Technology \\
E-mail: {\tt jakub@mimuw.edu.pl }
}

\author{Mariusz Niew\k{e}g\l owski}
\address{
Department of Applied Mathematics \\
Illinois Institute of Technology\\
Chicago, IL 60616, USA,
and
Faculty of Mathematics and
Information Science, Warsaw University of Technology\\
ul. Plac Politechniki 1, 00-661 Warszawa, Poland \\
E-mail:  {\tt
M.Nieweglowski@mini.pw.edu.pl}
}

\begin{document}
\maketitle

\begin{center}
First version circulated since 5 June 2011
\end{center}
\begin{abstract}
In this paper we investigate jump-diffusion processes in random environments which are given as the weak solutions to SDE's. We formulate conditions ensuring existence and uniqueness
in law of solutions. We investigate Markov property.
To prove  uniqueness we solve a general martingale problem for \cadlag processes. This result is of independent interest. Application of our results to a generalized exponential Levy model and a semi-Markovian regime switching model are presented in the last section.
\end{abstract}

\noindent
\begin{quote}


\end{quote}

\section{Introduction}

In this paper we investigate properties of stochastic differential equation (SDE)  which describes a behavior of some system in a random
environment in the time  interval  $[\![ 0, T ]\!]$, $T < \infty$. We
model this behavior by a process $Y$ being jump-diffusion which is a solution to some SDE. As an example of SDE
considering in this paper, an SDE driven
by a L\'{e}vy process $Z$ and some counting processes can be taken. This SDE has a nice interpretation so we describe it in detail. Let $Z$ be an
$n$-dimensional L\'{e}vy process with the L\'{e}vy measure $\nu$
satisfying
\begin{equation}\label{eq:nu-int-cond}
    \int_{\mathbf{R}^n} \big( \norm{x}^2 \wedge 1 \big) \nu(dx) < \infty.
\end{equation}
By $\mathcal{K} = \set{1 , \ldots , K }$ we denote the set of indices of
environments in which our system can stay. The set $\mathcal{K}$ can describe states of hybrid model, states of economy, rating classes in modeling credit risk etc.
 A stochastic process $C$ with values in
$\mathcal{K}$ pointing out a type of  environment in which our system lives.
We assume that $C$ is a c\`{a}dl\`{a}g  process. Every change of $C$ results in  change
of drift and volatility of an SDE. Moreover, every jump of $C$ from
$i$ to $j$ results in  a jump of size ${\rho^{i,j}}$ of a system.
Jumps of $C$ are described by SDE driven by counting processes. Therefore, the evolution of $(Y,C)$ can be described as a solution of the
following SDE in $\mathbf{R}^{d} \times \mathcal{K}$:
\begin{equation}\label{eq:SDE-1}
\left\{
  \begin{array}{ll}
 d Y_t & =   \mu(t, Y_{t-} , C_{t-}) dt  +  \sigma( t, Y_{t-
}, C_{t-}) d Z_t
  + \sum_{ \substack{i,j=1 \\ j\neq i} }^K \rho^{i,j}(t,Y_{t-}) \I_\set{i}(C_{t-}) d N^{i,j}_t,
\\
 d C_t & = \sum_{ \substack{i,j=1 \\ j\neq i} }^K (j - i) \I_{i} (C_{t-}) d
 N^{i,j}_t ,
 \\
 Y_0 & = y , \quad C_0 = c \in \mathcal{K}.
  \end{array}
\right.
\end{equation}
For fixed $i,j \in \mathcal{K}$, the process $N^{i,j}$  is a counting point
process with intensity $\lambda^{i,j}(t,Y_{t-})$, and
\begin{flalign*}
(\mathbf{\Lambda}) &&
 \lambda^{i,j}(\cdot, \cdot) : [\![0,T]\!] \times
\mathbf{R}^d \rightarrow \mathbf{R}_+ \text{ is a bounded
continuous function in } (t, y).
\end{flalign*}
It means that, for the fixed  $(i,j)$, the process
\begin{align}\label{eq:mart-Mij}
M^{i,j}_t := N^{i,j}_t - \int_{]0,t]} \lambda^{i,j}(u, Y_{u-}) du
\end{align}
is a martingale. Moreover, we require that $ N^{i,j}_0=0$ and processes $Z$,
$N^{i,j}$, $i,j \in \mathcal{K}$, $i \neq j$, have no common
jumps, i.e.,
  \begin{equation}
 \Delta Z_t \Delta N^{i,j}_t =0 \quad  \mathbf{P} - a.s.,
\end{equation}
and for all $(i_1,j_1) \neq ( i_2,j_2)$ we have
\begin{equation}
    \Delta N^{i_1,j_1}_t\Delta N^{i_2,j_2}_t =  0 \quad  \mathbf{P} - a.s.
\end{equation}
The coefficients in SDE \eqref{eq:SDE-1} are measurable functions
$\mu(\cdot, \cdot,\cdot): [\![0,T]\!] \times \mathbf{R}^d \times
\mathcal{K} \rightarrow \mathbf{R}^d $, $\sigma( \cdot, \cdot,
\cdot ) : [\![0,T]\!] \times \mathbf{R}^d \times \mathcal{K}
\rightarrow \mathbf{R}^{d\times n} $ and $\rho^{i,j}(\cdot,\cdot)
: [\![0,T]\!] \times \mathbf{R}^d  \rightarrow \mathbf{R}^d$. This
SDE is a non-standard one since driving noise depends on the
solution itself, similarly as in Jacod and Protter \cite{jp1978}
or in Becherer and Schweizer \cite{bs2005}. That is, the noise
$(N^{i,j})_{i,j \in \mathcal{K} : j\neq i}$ is not given apriori
and is also constructed, so the solution is not a pair $(Y,C)$ but
a quadruple $(Y,C, (N^{i,j})_{i,j \in \mathcal{K} : j\neq i},Z)$.


Therefore it is interesting to find a weak solution to
\eqref{eq:SDE-1} and to formulate conditions ensuring uniqueness
in law of solutions. In our paper we investigate more general jump-diffusion processes than SDE \eqref{eq:SDE-1}. We construct a weak solution 
to this general SDE \eqref{eq:SDE-gen} (Thm. \ref{thm:sc-construction}) and show Markov property of components $(Y,C)$ of this solution (Thm. \ref{thm:markov-prop}).
 In Section 3, we show  finiteness of $2m$-moments of
$\sup_{t \in [\![0,T]\!]} |Y_t|$ for $Y$ being component of any weak solution to \eqref{eq:SDE-gen}.
Our goal is to prove uniqueness in law of solutions to SDE \eqref{eq:SDE-gen}.
To do this we use the notion of martingale problem.
So, in Section 4, we consider and solve a general martingale problem for \cadlag processes. Results of this section are of independent interest.
 In Section 5, using the notion of martingale problem we solve the problem of uniqueness in law of solutions to \eqref{eq:SDE-gen} under some natural condition imposed on coefficients of SDE and intensity. Knowing that the components $(Y,C)$ of solution solve local martingale problem (Prop. \ref{prop:mart-prob}) and that for any solution  of local martingale problem holds
$
\mathbb{E} \left(  \sup_{ t \leq T} |Y_t|^{2m}\right) < \infty
$ (Thm. \ref{prop:moment-mp}), we prove in Theorem  \ref{prop:localmp-mp} that a solution of local martingale problem is a solution of martingale problem. Next, we prove that  a martingale problem is well-posed (Thm. \ref{thm:markov-mp}) which  implies uniqueness in law of solution to \eqref{eq:SDE-gen}.
Application of our results is  presented in the last section on an example of semi-Markovian regime switching model and  on an example of generalized exponential Levy model. This model is very useful in finance, see  Cont and Voltchkova \cite{convol2005a} or Jakubowski and Niew\k{e}g\l owski \cite{jaknie2012}.
The model considered in our paper is related to the regime switching diffusion models with state-dependent switching which were considered in Becherer and Schweizer \cite{bs2005}, Yin and Zhu \cite{yinzhu2010}, \cite{yinzhu2010a}, Yin, Mao, Yuan and Cao \cite{yinmaoyc2010} amongst others. Our model generalize jump-diffusions with state dependent switching introduced by Xi and Yin \cite{xiyin2011}, \cite{yinxi2010} and studied further by Yang  and Yin \cite{yangyin2012}. Our main contribution is the presence of functions $\rho^{i,j}$ in  SDE \eqref{eq:SDE-1}, which allows to model the fact that the process $Y$ jumps at the time of switching the regime $C$.
 This is very important from the point of view of applications, because it adds extra flexibility to the model.
 For example, it gives us a possibility of introducing the dependence of intensity of jumps $C$ at time $t$ on the trajectory of process $C$ up to time $t-$. This is the case of semi-Markov processes, where $\lambda^{i,j}$ at time $t$ depends on time that process $C$ spends in current state after the  last jump.
The process (say $Y^1$) corresponding to this semi-Markovian dependence can be introduced in our framework by setting
\[
	d Y^1_t = dt - \sum_{i,j \in \mathcal{K}: i \neq j } Y^1_{t-} \I_{i}(C_{t-}) d N^{i,j}_t.
\]
Hence, allowing  $\lambda^{i,j}$ to be a (non-constant) function of $Y^1$ we obtain a semi-Markov model.

\section{Solutions to SDE's defining jump-diffusions in random environments}
\subsection{Formulation of problem}

Fix $T < \infty$. We investigate a weak solution to a SDE on a time interval $[\![r,T]\!]$, $r < T$, given by
\begin{equation}\label{eq:SDE-gen}
\left\{
  \begin{array}{ll}
 d Y_t &= \mu(t, Y_{t-} , C_{t-}) dt  +  \sigma( t, Y_{ t-
}, C_{t-}) d W_t + \int_{\norm{x} \leq a } F(t,Y_{t-},C_{t-}, x )
\widetilde{\Pi}(dx , dt)
\\
& \
  + \int_{\norm{x} > a } F(t,Y_{t-},C_{t-}, x ) \Pi(dx , dt)  + \sum_{ \substack{i,j=1 \\ j\neq i} }^K \rho^{i,j}(t,Y_{t-}) \I_\set{i}(C_{t-}) d N^{i,j}_t,
\\
& \\
 d C_t &= \sum_{ \substack{i,j=1 \\ j\neq i} }^K (j - i) \I_{i} (C_{t-}) d
 N^{i,j}_t ,
 \\
 Y_r & = y , \quad C_r = c \in \mathcal{K},
  \end{array}
\right.
\end{equation}
where $W$ is a standard $p$-dimensional Wiener process, $a > 0$ is
fixed, $\Pi(dx,dt)$ is a Poisson random measure on
$\mathbf{R}^n \times [\![r,T]\!]$ with the intensity measure
$\nu(dx) dt $, $\nu$ is a Levy measure, and
$N^{i,j}$  are counting point processes with intensities determined by
$\lambda^{i,j}$, bounded nonnegative continuous functions in $(t, y)$,
such that processes defined by \eqref{eq:mart-Mij} are martingales.
By $\widetilde{\Pi}$ we denote the compensated measure of $\Pi$, i.e.,
 \[
    \widetilde{\Pi}(dx , dt) := \Pi (dx , dt) - \nu(dx )dt.
 \]
The coefficients in SDE \eqref{eq:SDE-gen} are measurable, locally bounded, deterministic
functions $\mu(\cdot,\cdot, \cdot): [\![0,T]\!] \times
\mathbf{R}^d \times \mathcal{K} \rightarrow \mathbf{R}^d $,
$\sigma( \cdot, \cdot, \cdot ) : [\![0,T]\!] \times \mathbf{R}^d
\times \mathcal{K} \rightarrow \mathbf{R}^{d\times p} $,
$\rho^{i,j}(\cdot,\cdot) : [\![0,T]\!] \times \mathbf{R}^d
\rightarrow \mathbf{R}^d$, and
$F(\cdot,\cdot, \cdot, \cdot ) :
[\![0,T]\!] \times \mathbf{R}^d \times \mathcal{K} \times
\mathbf{R}^n \rightarrow \mathbf{R}^{d} $  is such that the mapping
\begin{align}\label{eq:F2-loc-bound}
(t,y,c) \longmapsto
\int_{\norm{x} \leq a
}\!\!\left| F(t,y,c,x) \right|^2\nu(dx)\!
\end{align}
is locally bounded. Moreover, we require that the
Poisson random measure $\Pi $ and the processes $N^{i,j}$, $i,j
\in \mathcal{K}$, $i \neq j$, have no common jumps, i.e., for
every $t > r > 0 $ and every $b>0$,
  \begin{equation} \label{eq:hip-A}
\int_r^t \int_{ \norm{ x} > b}
 \Delta N^{i,j}_u \Pi( dx, du) =0 \quad  \mathbf{P} - a.s.
\end{equation}
and for all $(i_1,j_1) \neq ( i_2,j_2)$,
 \begin{equation}
\label{eq:hip-B}
    \Delta N^{i_1,j_1}_t\Delta N^{i_2,j_2}_t =  0 \quad  \mathbf{P} - a.s.
\end{equation}
SDE \eqref{eq:SDE-gen} is a generalization of SDE \eqref{eq:SDE-1}, but has no such simple interpretation as
\eqref{eq:SDE-1}. However \eqref{eq:SDE-1}  allows to describe the more complex systems.
If
\[ F(t,y,c,x) = \sigma(t,y,c) x,
\]
then SDE \eqref{eq:SDE-gen}  takes the form \eqref{eq:SDE-1}. The
following   processes play an important role in our considerations:
\begin{align}\label{eq:Hi}
    H^i_t := \I_\set{ C_t = i},
\end{align}
\begin{align}\label{eq:Hij}
    H^{i,j}_t := \sum_{ r < u \leq t} \I_\set{ C_{u-} = i}\I_\set{ C_{u} = j} =\sum_{ r < u \leq t} H^{i}_{u-}H^{j}_{u}
\end{align} for $i,j \in \mathcal{K}$, $ i \neq j$. The random variable  $H^i_t$ indicates a state in which $C$ is at the moment $t$, and $ H^{i,j}_t$ counts a number of jumps of $C$ from $i$ to $j$ up to time $t$.

\begin{rem}
To distinct solutions started  from different $y,c$ at different times $r\in [\![0,T]\!]$ it is  convenient to denote by $(Y^{r,y,c}_t,C^{r,y,c}_t)_{t \in [\![r,T]\!]}$ a solution to SDE \eqref{eq:SDE-gen} started from $(y,c)$ at time $r$.  Sometimes,
for a notational convenience we drop $r,y,c$ in this notation if there is no confusion.
By $\mathbf{P}_{r,y,c}$ we denote the law of $(Y^{r,y,c}_t,C^{r,y,c}_t)_{t \in [\![r,T]\!]}$.
\end{rem}

\subsection{Existence of  weak solutions to SDE's
in random environments}

We prove the existence of a weak solution to SDE \eqref{eq:SDE-gen}
  using an argument of a suitable change of
measure (cf. \cite{bs2005} or Kusuoka \cite{kus1999}). For a
matrix $A \in \mathbf{R}^{d\times d}$, by $\norm{ A }$ we denote
the matrix norm given by
\[
    \norm{ A }^2 := \sum_{i=1}^d \sum_{j=1}^d \abs{a_{i,j}}^2,
\]
and for simplicity of notation we use $N$ for $(N^{i,j})_{i,j \in \mathcal{K} : j\neq i}$.
\begin{thm}\label{thm:sc-construction}
Assume that coefficients $\mu$, $\sigma$, $F$  satisfy conditions:

\noindent a) the linear growth condition: there exists a constant $K_1>0$ such that
\begin{flalign*}
(\mathbf{LG}) && \left| \mu(t,y,c) \right|^2 + \norm{
\sigma(t,y,c) \sigma(t,y,c)^\top } + \int_{\norm{x} \leq a } \abs{
F(t,y,c,x) }^2 \nu(   dx) \leq K_1 ( 1 + \abs{y}^2),&&
\end{flalign*}
\noindent b) the Lipschitz condition: there exists a constant $K_2>0$ such that
\begin{flalign*}
(\mathbf{Lip}) && &
 \left| \mu(t,y_1,c) - \mu(t,y_2,c)  \right|^2 +
\norm{ \sigma(t,y_1,c)  - \sigma(t,y_2,c) }^2 & \\
&&& + \int_{\norm{x} \leq a} \abs{  F(t,y_1,c,x) - F(t,y_2 ,c,x) }^2
\nu(dx)
 \leq K_2 \abs{y_1 - y_2 }^2, &
\end{flalign*}
\noindent c)  for every $c \in \mathcal{K}$ and every $k \in
\mathcal{K}\setminus \set{c}$
\begin{flalign*}
(\mathbf{Cont}) &&  \!\!\! (t, y) \rightarrow F(t,y,c,x) \quad
\text{for $\norm{x} > a  $},\ and \quad (t, y) \rightarrow
\rho^{c,k}(t, y) & \text{ are continuous. }
\end{flalign*}
Then there exists a weak solution $(Y,C, W, \Pi, N)$ to SDE \eqref{eq:SDE-gen} on $[\![r,T]\!]$, which is
adapted, c\`{a}dl\`{a}g, and moreover \eqref{eq:hip-A} and
\eqref{eq:hip-B} are satisfied.
\end{thm}

\begin{proof}
In the first step we consider a filtered probability space
$(\Omega, \mathcal{F}, \mathbb{F},\mathbf{P})$ on which there
exist independent processes: a standard Brownian motion $W$,
a Poisson random measure $\Pi(dx, dt)$  with intensity measure
$\nu(dx)dt$
 and the Poisson processes $(N^{i,j})_{i,j \in
\mathcal{K} : i \neq j }$  with intensities equal to one.
Let us consider, on this stochastic basis, a SDE for $[\![r,T]\!]$:
\begin{align}
&&\nonumber d Y_t &= \mu( t, Y_{t-} , C_{t-}) dt  + \sigma( t, Y_{t-
}, C_{t-})dW_t + \!\int_{\norm{x} \leq a } \!\!\!\!F(t,Y_{t-},C_{t-}, x )
\widetilde{\Pi}(dx , dt)
\\
\label{eq:SDE-2}  &&&
  \ \ \ + \!\int_{\norm{x} > a } \!\!\!\!F(t,Y_{t-},C_{t-}, x ) \Pi(dx , dt)
  + \sum_{\substack{i,j \in \mathcal{K} \\ j\neq i } } \rho^{i,j}(t,Y_{t-}) \I_\set{i}(C_{t-})  d N^{i,j}_t,
\\
&&
\nonumber
 d C_t &= \sum_{ \substack{ i, j \in \mathcal{K} \\  j  \neq i } } (j - i) \I_{i} (C_{t-}) d
 N^{i,j}_t
\end{align}
with $Y_r = y$, $C_r = c \in \mathcal{K}$. For this SDE  results
on existence and uniqueness of solutions exist (see, e.g.,  Applebaum
\cite[Thm. 6.2.9]{apple2004}). So, under our assumptions, there
exists the unique strong solution $(Y,C)$ to  SDE
\eqref{eq:SDE-2} which is adapted and c\`{a}dl\`{a}g. Let us
define a new measure $\mathbf{P}^{\lambda,r}$ by
\begin{equation}\label{eq:density-p-lambda}
\frac{ d \mathbf{P}^{\lambda,r} }{ d \mathbf{P}}=  Z_T :=
\mathcal{E}_T \left( \sum_{\substack{ i,j \in \mathcal{K}  \\ j \neq i } } \int_{]\!]r, \cdot]\!]} (
\lambda^{i,j}(t,Y_{t-}) - 1) ( d N^{i,j}_t - dt )\right),
\end{equation}
where $\mathcal{E}$ denotes the Doleans-Dade exponential.
The assumption that $\lambda^{i,j}$ are non-negative bounded measurable functions
for $i,j \in \mathcal{K} $, $i \neq j $
implies that $\mathbf{P}^{\lambda,r}$  is a probability measure
(see Br\'{e}maud
\cite[Thm. VI.2.T4]{bremaud1981}).
 Moreover, for each
$i,j$ the standard Poisson process $(N^{i,j})_{t \in [\![r,T]\!]}$ under $\mathbf{P}$ is
a counting process under $\mathbf{P}^{\lambda,r}$ with the
$\mathbf{P}^{\lambda,r}$-compensator $(\int_{]r,t]} \lambda^{i,j}(u,
Y_{u-}) du)_{t \in [\![r,T]\!]}$  (see \cite[Thm.
VI.2.T2]{bremaud1981}). Hence, we see that if $(Y,C)$ is a
solution to \eqref{eq:SDE-2}, then the process $(Y,C, W, \Pi,
N)$ solves
\eqref{eq:SDE-gen} under $\mathbf{P}^{\lambda,r}$.
 It remains to prove that processes   have no common jumps, i.e. \eqref{eq:hip-A} and \eqref{eq:hip-B}.
Let, for a fixed $b>0$,
\[
A^{i,j}_b := \set{ \omega: \exists t \in [\![r,T]\!] \ \ such \ that \ \
\int_{r}^t \int_{\norm{x} > b } \Delta N^{i,j}_u(\omega) \Pi(du,dx
) \neq 0 }.
\]
Independence of $N^{i,j}$ and $\Pi$ under $\mathbf{P}$ yields that
$\mathbf{P}(A^{i,j}_b) =0 $ for each $b > 0$. Absolute continuity
of $\mathbf{P}^{\lambda,r} $ with respect to $\mathbf{P}$ implies that
$\mathbf{P}^{\lambda,r}(A^{i,j}_b) =0 $, so $N^{i,j}$ and $\Pi$ have
no common jumps. In an analogous way we see that the processes $N^{i,j}$ and
$N^{k,l}$ have no common jumps for $(i,j)\neq (k,l)$.
 \end{proof}

\subsection{Markov property of solutions }

At first, we are interested in Markov property of solutions to SDE
\eqref{eq:SDE-gen}. Note that in the formulation of  SDE
coefficients depend only on values of $(Y, C)$ at time $t-$, and
since a driving noise can be considered as Poissonian type we can
expect that a solution of this SDE possesses a Markov property. We
start from the study of solution constructed in Theorem
\ref{thm:sc-construction}.
\begin{thm}\label{thm:markov-prop}
The components $(Y,C)$ of the solution to
the SDE \eqref{eq:SDE-gen} constructed in Theorem
\ref{thm:sc-construction}
have Markov property.
 \end{thm}
\begin{proof}
Let $(Y,C,W,\Pi, (N^{i,j})_{i ,j \in K : j \neq i})$ on
$(\Omega,\mathcal{F} , \mathbb{F},\mathbf{P}^{\lambda,r} )$ be the
solution to SDE \eqref{eq:SDE-gen} on $[\![r,T]\!]$ constructed in Theorem
\ref{thm:sc-construction}. Fix $t \geq r  $, and let $X $ be an arbitrary
bounded random variable measurable with respect to
$\sigma((Y_u,C_u) : T \geq u \geq t)$. To prove Markovianity it is sufficient to check
that there exists a measurable function $f$  such that
\begin{align}\label{eq:markov}
\mathbf{E}_{\mathbf{P}^{\lambda,r}}( X | \mathcal{F}_t )= f(t,Y_t , C_t).
\end{align}
The density process defined by the Doleans-Dade exponential
\eqref{eq:density-p-lambda} can be written explicitly (see Protter
\cite[Thm. II.8.37]{prot2004}) as
\[
    Z_t = \prod_{ \substack{ i,j \in \mathcal{K} \\  j \neq i }}   \left( e^{\int_r^t (\lambda^{i,j}(u, Y_{u-}) -1 ) du }
     \prod_{ r < u \leq t} \left(  1 +   (\lambda^{i,j}(u, Y_{u-}) -1 ) \Delta N^{i,j}_u \right)
     \right) .
\]
For convenience we introduce the following notation
\[
    L_{s,t} := \prod_{ \substack{ i,j \in \mathcal{K} \\  j \neq i } }   \left( e^{\int_s^t (\lambda^{i,j}(u, Y_{u-}) -1 ) du }
     \prod_{ s < u \leq t} \left(  1 +   (\lambda^{i,j}(u, Y_{u-}) -1 ) \Delta N^{i,j}_u \right)
     \right),
\]
which allows to write the density in the shorter way
\[
    Z_T = L_{r,T} = L_{r,t} L_{t,T} = Z_t L_{t,T}.
\]
Using similar arguments as in the proof of abstract Bayes formula
we get
\begin{align}\label{eq:bayes}
&
\I_\set{Z_t > 0}
\mathbf{E}_{\mathbf{P}^{\lambda,r}}( X | \mathcal{F}_t ) =
\I_\set{Z_t > 0}
\mathbf{E}_{\mathbf{P}} \left( X  L_{t,T}| \mathcal{F}_t \right)
= f(t,Y_t, C_t) \I_\set{Z_t > 0},
\end{align}
where the last equality follows from the fact that the random variable $X L_{t,T}$ is
$\mathbf{P}$ integrable and measurable with respect to
$\sigma((W_u - W_t, \Pi(A,[t,u]), (N^{i,j}_u - N^{i,j}_t)_{i,j \in
\mathcal{K} : j\neq i}) : A \in \mathcal{B} (\mathbf{R}^n ): T
\geq u \geq t)$ (see \cite[Thm. V.6.32]{prot2004}). Equality
\eqref{eq:bayes} implies \eqref{eq:markov} since
$\mathbf{P}^{\lambda,r}( Z_t = 0 ) = 0$.
\end{proof}
Theorem \ref{thm:markov-prop} suggests that for any arbitrary
solution to
SDE \eqref{eq:SDE-gen} the process $(Y,C)$ is a time inhomogenous
Markov process. We prove further,  that this suggestion is true.


In the sequel we  will use frequently the following technical result giving the  canonical decomposition of a special semimartingale $(v(t,Y_t,C_t))_{t \in [\![r,T]\!]}$, which plays a crucial role in what follows.
This result is a consequence of  It\^{o}'s lemma for general semimartingales (see \cite[Thm. II.7.32]{prot2004}).  Let us denote by  $C^{1,2}  =
C^{1,2}( [\![0,T]\!] \times \mathbf{R}^d  \times \mathcal{K})$   -
the  space of all measurable functions $v : [\![0,T]\!] \times
\mathbf{R}^d \times \mathcal{K} \rightarrow \mathbf{R} $ such that
$ v(\cdot, \cdot, k ) \in  C^{1,2}([\![0,T]\!] \times \mathbf{R}^d
)$  for every $ k \in \mathcal{K} $, and let $C^{1,2}_c$ be a set
of functions $f \in C^{1,2}$ with compact support.

\begin{thm}\label{thm:ito}
Let $(Y,C,W,\Pi,N)$ be a solution to SDE \eqref{eq:SDE-gen} on $[\![r,T]\!]$ and  $ v \in
C^{1,2}$  be a function such that the mapping
\begin{align}\label{eq:ito-formula-int-cond}
(t,y,c) \mapsto \int_{\mathbf{R}^n}\!\!\left| v( t, y  + F(t,y,c,x) ,c) - v( t, y,c) - \nabla v(t, y,c)\!F(t,y,c,x) \I_{\set{ \norm{x} \leq a
}}\right|\nu(dx)
\end{align}
is locally bounded (i.e., bounded on compact sets),
where $\nabla v$ denotes the vector of partial derivatives of $v$ with respect to components of $s$. Then the process $(v (t,Y_t, C_t))_{t \geq r}$
 is a special semimartingale with the following (unique) canonical decomposition 	
\begin{align}
\nonumber d v (t,Y_t, C_t) &= (\partial_t  + \mathcal{ A}_t ) v(t,Y_{t-} ,C_{t-}) dt
+ \sum_{i} H^i_{t- } \nabla v(t, Y_{t-},i) \sigma(t, Y_{t-}, i) d
W_t
\\
\label{eq:ito-formula} &    + \sum_{i} H^i_{t- }
 \int_{\mathbf{R}^n} (v(t , Y_{t-} +  F(t,Y_{t-}, i, x) ,i) - v(t , Y_{t-},i)) \widetilde{\Pi} (dx,dt)
\\
\nonumber &  + \sum_{i,j:j \neq i } (v(t , Y_{t-} +
{\rho^{i,j}}(t,Y_{t-}),j) - v(t , Y_{t-},i) ) H^{i}_{t-}d M^{i,j}_t,
\end{align}
where $H^i_t = 1_\set{i} (C_t) $ and $\mathcal{A}_t$ is defined by
\allowdisplaybreaks
\begin{align}
\label{eq:gener-SDE}&\mathcal{ A }_t v(y,c) :=
  \nabla v(t,y,c)\mu(t,y,c)
+ \frac{1}{2}
 \Tr \left( a(t,y,c) \nabla^2 v(t,y,c) \right)
\\ \nonumber
 & + \int_{\mathbf{R}^n}\left( v( t, y  + F(t,y,c,x) ,c)
- v( t,y,c) - \nabla v(t,y,c) F(t,y,c,x)\I_{\set{\norm{x} \leq a}}
\right)\nu(dx)
\\
\nonumber & + \sum_{k \in \mathcal{K} \setminus c } \left( v( t, y +
\rho^{c,k}(t, y),k) - v(t,y,c) \right)\lambda^{c,k}(t, y).
\end{align}
Here $a(t,y,c) :=  \sigma(t,y,c)  (\sigma(t,y,c))^{\!\top}$\!, and
by $Tr$ we denote the trace operator, 
and $\nabla^2 v$ is the matrix of second derivatives of $v$ with respect to the
components of $s$.
\end{thm}
\noindent The proof of this technical result is given in the
appendix. From the above theorem we obtain a nice martingale property:
\begin{prop}\label{prop:mart-prob}
Let $(Y,C)$ be components of a solution to SDE
\eqref{eq:SDE-gen} on $[\![r,T]\!]$. Then for any $f \in C^{1,2}_c$ the process
    \[
        M^f_t   := f(Y_t, C_t)  - \int_r^t \mathcal{A}_u f( Y_{u-}, C_{u-})  du
    \]
is an $\mathbb{F}$--local martingale on $[\![r,T]\!]$.
\end{prop}
\begin{proof}
Local boundedness of function \eqref{eq:F2-loc-bound} implies \eqref{eq:ito-formula-int-cond} for $f \in C^{2}_c$,
so \eqref{eq:ito-formula} holds. Therefore the process $M^f$ is a
local martingale.
\end{proof}
Since $M^f$ is $\mathbb{F}^{Y,C}$-adapted, we obtain immediately
\begin{cor}\label{cor:mart-prob}
Under assumptions of Proposition \ref{prop:mart-prob} the process
$M^f$ is an $\mathbb{F}^{Y,C}$--local martingale for any $f \in
C^{2}_c$.
\end{cor}

\section{Moment estimates}

In this section we prove  finiteness of $2m$-moments of
$\sup_{t \in [\![0,T]\!]} |Y_t|$, which  is the crucial fact in the proof of uniqueness in law of solutions to SDE \eqref{eq:SDE-gen}.
We stress that in this section we do not assume that $(Y,C)$ solves SDE but only that $(Y,C)$ solves a local martingale problem corresponding to the generator of \eqref{eq:SDE-gen}.
Throughout
this section we make the following additional mild assumption: \\
\textbf{(LB)} The mapping
    \begin{align}\label{eq:loc-bound}
        (t,y,c) \rightarrow \int_{ |F(t,y,c,x)| > 1 } |F(t,y,c,x)| \nu(dx)
    \end{align}
    is locally bounded.
%

In what follows we use
the notation $z^{\!\top} =(z^{\!\top}_1, z_2)$, where
$z_1 \in \mathbf{R}^d $ corresponds to coordinates of $Y$ and $z_2 \in
\mathbf{R}$ corresponds to $C$.
Let us introduce the following functions
\begin{align}
\label{eq:tilde-b}
&\widetilde{b}(t,y,c) := \mu(t,y,c) +
 \!\! \int_{\mathbf{R}^n} F(t,y,c,x)\left( \I_{\set{\norm{F(t,y,c,x)} \leq 1}} -
 \I_{\set{\norm{x}\leq a}} \right) \nu(dx),
 \\
\nonumber
 &a(u,y,c) = \sigma \sigma^{\top} (u,y,c),
\end{align}
and measure $\overline{\nu}(t,y,c, \cdot )$  on $\mathbb{R}^d \times \mathbb{R}$ defined by
\begin{align}\label{eq:komp-sc}
\overline{\nu}(t,y,c,dz_1, d z_2) &:=
\nu_F(t,y,c, d z_1) \otimes \delta_\set{0}(dz_2) \\
\nonumber & \quad + \I_\mathcal{K}(c)\sum_{ k \in \mathcal{K} \setminus \set{c}}
\lambda^{c,k}(t, y) \delta_\set{ (\rho^{c,k}(t, y) ,
k-c)} (dz_1, dz_2),
\end{align}
where
$\nu_F(t,y,c,\cdot)$ is a measure defined for $A \in \mathcal{B}(\mathbf{R}^d)$  by setting
\[
    \nu_F( t,y,c, A ) := (\nu \circ F^{-1}_{t,y,c})(A)
    =
    \nu(\set{x : F(t,y,c,x) \in A}),
\]
 $\delta_a$ denotes Dirac measure at $a$.
\begin{thm}\label{thm:semimart-dec}
    Assume that the law of a process $(Y,C)$ solves the local martingale problem for $(\mathcal{A}_t)_{t \in [\![r,T]\!]}$ given by \eqref{eq:gener-SDE}.
	Then \\
i) The process $(Y,C)$  is a semimartingale with the characteristics $(\widetilde{B}, \widetilde{C}, \widetilde{\nu})$, with respect to the truncation
function $h(z) := z
\I_\set{ \norm{z_1} +  2|z_2| \leq 1 }$, of the form
\begin{align*}
    \widetilde{B}_t = \int_{0}^t \left[
                       \begin{array}{c}
                         \widetilde{b}(u,Y_{u-},C_{u-}) \\
                         0 \\
                       \end{array}
                     \right]
    du,
    \qquad
    \widetilde{C}_t = \int_{0}^t
    \left[
    \begin{array}{cc}
      a(u, Y_{u-}, C_{u-}) & \mathbf{0} \\
      \mathbf{0}^\top & 0
    \end{array}
    \right]
    du,
\end{align*}
\begin{align}\label{eq:komp-sc-exact}
    \widetilde{\nu}([\![0,t]\!]\times A_1 \times A_2)
    :=
    \int_{]\!]0,t]\!]} \int_{A_1 \times A_2} \overline{\nu}(u,Y_{u-}, C_{u-}, dz_1, d z_2) du.
\end{align}
ii) The process $(Y,C)$ has the following decomposition
\begin{align*}
    \left[
    \begin{array}{c}
      Y_t \\
      C_t \\
    \end{array}
    \right]
    =
    \left[
    \begin{array}{c}
      Y_0 \\
      C_0 \\
    \end{array}
    \right]
    +
    \tilde{B}_t
    +
    \left[
    \begin{array}{c}
      Y^c_t \\
      0 \\
    \end{array}
    \right]
    +
    \int_{]\!]0,t]\!]} \int_{\mathbf{R}^{d+1}}
    \!\!\!
      h(z)
    \widetilde{\mu}(du,dz)
    +
    \int_{]\!]0,t]\!]} \int_{\mathbf{R}^{d+1}}
    \!\!\!
      (z - h(z))
    \overline{\mu}(du,dz),
\end{align*}
where $Y^c$ is the martingale continuous part of $Y$,
$\overline{\mu}(dt,dz)$ is the measure associated with jumps of $(Y,C)$, and
$\widetilde{\mu}(dt,dz) := \overline{\mu}(dt,dz) -
\overline{\nu}(dz)dt$ is the compensated measure of jumps.
\end{thm}
\begin{proof}
By our assumption,  for   $v \in C^{2}_c$, 
the   function
\begin{align*}
I_v(t,y,c):=\int_{\mathbf{R}^n}\left( v( y  + F(t,y,c,x) ,c) -
v(y,c) - \nabla v(y,c)
F(t,y,c,x)\I_{\set{\norm{F(t,y,c,x)}\leq 1}} \right)\nu(dx)
\end{align*}
is bounded on compact sets. We can rewrite the generator
$\mathcal{ A }_t $ in the   form
\begin{align*}
&\mathcal{ A }_t v(y,c) =
\nabla v(y,c)\widetilde{b}(t,y,c)
+ \frac{1}{2}
 \Tr \left( a(t,y,c) \nabla^2 v(y,c) \right)
\\ \nonumber
 & + \int_{\mathbf{R}^n}\left( v( y  + F(t,y,c,x) ,c)
- v(y,c) - \nabla v(y,c)
F(t,y,c,x)\I_{\set{\norm{F(t,y,c,x)} \leq 1}} \right)\nu(dx)
\\
\nonumber & + \sum_{k \in \mathcal{K} \setminus c } \left( v(
s + \rho^{c,k}(t, y),k) - v(y,c) \right)\lambda^{c,k}(t, y),
\end{align*}
where $\widetilde{b}$ is given by \eqref{eq:tilde-b}. Note that $\widetilde{b}$ is also bounded on compact sets. Note that any function $v \in C^{2}_c(
\mathbf{R}^{d} \times \mathcal{K})$   can be  extended  to a
function 
$\widetilde{v} \in
C^{2}_c(\mathbf{R}^{d+1} )$.
Therefore, we can extend $\mathcal{A}_t$ to an operator $\widetilde{\mathcal{ A }}_t$  acting on functions $ v \in
C^{2}_c(\mathbf{R}^{d+1})$
by formula
\begin{align*}
&\widetilde{\mathcal{ A }} v(y,c) =
\nabla v(y,c) \widetilde{b}(t,y,c)
+ \frac{1}{2}
 \Tr \left( a(t,y,c) \nabla^2 v(y,c) \right)
\\ \nonumber
 & + \int_{\mathbf{R}^{n+1}}\left( v( y + z_1 , c+ z_2)
- v(y,c) -  \nabla_{y,c} v(y,c) h(z)  \right)
\overline{\nu}(t,y,c,dz)
\end{align*}
where $\overline{\nu}$ is given by \eqref{eq:komp-sc} and for $z^{\!\top}=(z_1^{\!\top},z_2)$, $z_1 \in
\mathbf{R}^d$, $z_2 \in \mathbf{R}$,
 \[ h(z) := z
\I_\set{ \norm{z_1} +  2|z_2| \leq 1 }, \qquad and \qquad \nabla_{y,c}
v(y,c) := \left[ \begin{array}{c}
                           \nabla v(y,c), \
                           \partial_c v(y,c)
                         \end{array}
                         \right].
\]
Indeed, this formula gives the extension
since by simple calculation we can see that for every $c
\in \mathcal{K}$,
\[
 \widetilde{\mathcal{A}}_t \widetilde{v}(y,c) =  \mathcal{A}_t
 v(y,c),
\]
for any extension $\widetilde{v}$ of $v$.
Hence, by our assumption, we have that the process
\[
    M^v_t := v(Y_t,C_t) - v(Y_0,C_0) - \int_{0}^t \widetilde{\mathcal{A}}_u v(Y_{u-},C_{u-})du
\]
is a local martingale for each $v \in
C^{2}_c(\mathbf{R}^{d+1} )$. Applying Theorem 1 from Griegolionis and Mikulevicius \cite{gregmik1981} (cf. Jacod and Shirayev \cite[Theorem II.2.42]{js1987}) we see that $(Y,C)$ is a semimartingale
with the characteristics, calculated with respect to the truncation
function $h$, of the form
\begin{align*}
    \widetilde{B}_t = \int_{0}^t \left[
                       \begin{array}{c}
                         \widetilde{b}(u,Y_{u-},C_{u-}) \\
                         0 \\
                       \end{array}
                     \right]
    du,
    \qquad
    \widetilde{C}_t = \int_{0}^t
    \left[
    \begin{array}{cc}
      a(u, Y_{u-}, C_{u-}) & \mathbf{0} \\
      \mathbf{0}^\top & 0
    \end{array}
    \right]
    du,
\end{align*}
\begin{align*}
    \widetilde{\nu}([\![0,t]\!]\times A_1 \times A_2)
    :=
    \int_{]\!]0,t]\!]} \int_{A_1 \times A_2} \overline{\nu}(u,Y_{u-}, C_{u-}, dz_1, d z_2) du.
\end{align*}
Therefore, by  Theorem II.2.34 \cite{js1987}, the process $(Y,C)$ has the canonical representation of
the form
\begin{align*}
    \left[
    \begin{array}{c}
      Y_t \\
      C_t \\
    \end{array}
    \right]
    =
    \left[
    \begin{array}{c}
      Y_0 \\
      C_0 \\
    \end{array}
    \right]
    +
    \tilde{B}_t
    +
    \left[
    \begin{array}{c}
      Y^c_t \\
      0 \\
    \end{array}
    \right]
    +
    \int_{]\!]0,t]\!]} \int_{\mathbf{R}^{d+1}}
    \!\!\!
      h(z)
    \widetilde{\mu}(du,dz)
    +
    \int_{]\!]0,t]\!]} \int_{\mathbf{R}^{d+1}}
    \!\!\!
      (z - h(z))
    \overline{\mu}(du,dz)
\end{align*}
where $Y^c$ is the martingale continuous part of $Y$,
$\overline{\mu}(dt,dz)$ is the measure of jumps of $(Y,C)$, and
\[
\widetilde{\mu}(dt,dz) := \overline{\mu}(dt,dz) -
\overline{\nu}(t,Y_{t-},C_{t-},dz)dt
 \]
 is the compensated measure of jumps. We use,
as before, the notation $z^{\!\top} =(z^{\!\top}_1, z_2)$, where
$z_1 \in \mathbf{R}^d $ corresponds to jumps of $Y$ and $z_2 \in
\mathbf{R}$ corresponds to jumps of $C$.
\end{proof}

\begin{thm}\label{prop:moment-mp}
    Assume that the law of a process $(Y,C)$ solves the local martingale problem for $(\mathcal{A}_t)_{t \in [\![r,T]\!]}$.
    Let $m$ be a natural number.     If the coefficients of $(\mathcal{A}_t)_{t \in [\![r,T]\!]}$ satisfy (LB), (LG),
and that   there exist
      constants $K_3$ and  $K_4$  such that for all $(t,y,c)$
\begin{align}
     \label{eq:LG-2}
     \int_{\norm{x} > a } | F(t,y,c,x)|^{2} \nu(dx)  + \big|\rho^{c,k}(t, y) \big|^2 \leq K_3(1 + |y|^2), \quad for \ k \in \mathcal{K}, k \neq c,
    \end{align}
 and
    \begin{align}
        \label{eq:LG-2m}
        \int_{ \mathbf{R}^n } | F(t,y,c,x)|^{2m} \nu(dx) \leq K_4(1 + |y|^{2m}),
    \end{align}
     then
    \begin{align}\label{eq:sup-S2}
        \mathbf{E} \left(\sup_{t \in [\![r,T]\!]} |Y_t|^{2m} \right)< \infty.
    \end{align}
    \end{thm}
\begin{proof}
We present a proof for $r=0$, the proof of general case is analogous.
In the first step we recall that,
 by Theorem \ref{thm:semimart-dec}, any solution  $(Y,C)$ to the martingale problem for $(\mathcal{A}_t)_{t \in [\![0,T]\!]}$ is a semimartingale. Since $h=(h_1, h_2)$, 	from \eqref{eq:LG-2} follows that
\[
		 \int_{]\!]0,t]\!]} \int_{\mathbf{R}^{d+1}}
      (z_1 - h_1(z))
    \overline{\nu}(u,y,c,dz) du
\]
 is well defined. Therefore
 the  semimartingale $Y$ is special and has  the following
canonical decomposition
\begin{align}\label{eq:repr-SC}
      Y_t
    &
    =
      Y_0
    +
    B_t
    +
      Y^c_t
    +
    \int_{]\!]0,t]\!]} \int_{\mathbf{R}^{d+1}}
      z_1
    \widetilde{\mu}(du,dz),
\end{align}
where
\[
    B_t  := \int_{0}^t b(u,Y_{u-}, C_{u-}) du,
    \qquad
    b(u,y,c):= \mu(u,y,c) + \int_{ \norm{x} > a } F(u,y,c,x) \nu(dx) + \sum_{k \in \mathcal{K} \setminus \set{c}} \rho^{c,k}(t, y) \lambda^{c,k}(t, y).
\]
Note that the function $b$ has also linear growth in $s$, since
$\mu$ has linear growth in $s$, by assumption. Thus, by the
Cauchy-Schwartz inequality and \eqref{eq:LG-2} we
have
\begin{align} \label{eq:warun}
   & \left|  \int_{ \norm{x} > a } F(u,y,c,x) \nu(dx)
    \right|^2
    \leq \nu(\norm{x} \geq a)\int_{ \norm{x} > a }  |F(u,y,c,x)|^2 \nu(dx)
    \\
    \nonumber
    &
        \leq \nu(\norm{x} \geq a ) K (1 + |y|^2).
\end{align}
    Now, let
    \[
        \tau_n := \inf \set{ t: \abs{Y_t} > n  }, \qquad
     Y^{*,2m}_t := \sup_{0 \leq u \leq t }
        \abs{Y_u}^{2m}.
    \]
    It is enough to show that there exists a constant $K$ such that
    \begin{align}\label{eq:gronwall-ineq}
        \mathbb{E} Y^{*,2m}_{t   \wedge \tau_n}
        \leq
         K\int_0^t \left( 1 + \mathbb{E} Y^{*,2m}_{s   \wedge \tau_n}        \right)  d s
    \end{align}
    for each $n$.
    Indeed, \eqref{eq:gronwall-ineq} and the Gronwall lemma imply
    \[
      \mathbb{E} Y^{*,2m}_{t   \wedge \tau_n} \leq K''(t),
    \]
    where $K''(t)$ does not depend on $n$, so using Fatou lemma we
    obtain \eqref{eq:sup-S2}, i.e.,
    \[
      \mathbb{E} Y^{*,2m}_{T} \leq K''(T) < \infty.
    \]
    Now we prove \eqref{eq:gronwall-ineq}. Using \eqref{eq:repr-SC} and applying the It\^o lemma to the function $|y|^{2m}$ and to the process $Y$ we obtain-
    \begin{align}
        \label{eq:S2m}
        &|Y_t|^{2m} =
          |Y_0|^{2m} + A^1_t + A^2_t +  M^c_t + M^{d}_t + D^1_t,
    \end{align}
    where
    \allowdisplaybreaks
    \begin{align*}
    &A^1_t
    := 2m \int_0^t |Y_{u-}|^{2m-2} Y_{u-}^{\!\top}b( u,  Y_{u-}, C_{u-})  du, \\
    &A^2_t :=
    \frac{1}{2}\int_0^t \left( 2m |Y_{u-}|^{2m-2}  Tr\left(a(u,Y_{u-},C_{u-})\right)  + 2m(2m-2) |Y_{u-}|^{2m-4} Tr \left( Y_{u-}Y_{u-}^{\!\top}a(u,Y_{u-},C_{u-})\right) \right) d u, \\
    &M^c_t := 2m \int_0^t |Y_{u-}|^{2m-2} Y_{u-}^{\!\top} d Y^c_u,
    \\
    &
    M^{d}_t :=
    2m \int_0^t \int_{\mathbf{R}^{d+1}} |Y_{u-}|^{2m-2} Y_{u-}^{\!\top} z_1 \widetilde{\mu}(du,dz),
    \\
    &
     \\
&    D^1_t :=
    \int_0^t
    \int_{\mathbf{R}^n} \left( |Y_{u-} + z_1 |^{2m} -
|Y_{u-}|^{2m} - 2m |Y_{u-}|^{2m-2} Y_{u-}^{\!\top}z_1 \right) \mu(du,dz),
\end{align*}
and $a(t,y,c) = \sigma(t,y,c) \sigma^{\!\top}\!(t,y,c)$.
We treat each of the components of $|Y_{t \wedge
\tau_n}|^{2m}$ separately. We start with $A^1$. The following
inequalities follow from the successive use of the Cauchy-Schwartz
inequality, Young inequality, \eqref{eq:warun}, and (LG) condition
\begin{align*}
        &|y|^{2m-2} \abs{y^{\!\top} b(u,y,c)}
        =
        |y|^{m}|y|^{m-2} \abs{y^{\!\top} b(u,y,c)}
        \leq
        \frac{1}{2}
        \left(
        |y|^{2m} + |y|^{2m-4} \abs{y^{\!\top} b(u,y,c)}^2
        \right)
        \\
        &
        \leq
        \frac{1}{2}
        \left(
        |y|^{2m} + K_1|y|^{2m-2} (1 + |y|^2)
        \right)
        \leq
        L_1
        \left( 1 + |y|^{2m}
        \right) ,
    \end{align*}
    and yields that
    \begin{align*}
        \mathbb{E}
        \sup_{ s  \leq t \wedge \tau_n }
        \abs{A^1_s }
       \leq
    L_1
    \mathbb{E}
    \int_0^{t  \wedge \tau_n  } \left(1+  |Y_{u-}|^{2m}  \right)du
        \leq
    L_1
    \int_0^t \left(1+  \mathbb{E} Y^{*,2m}_{u \wedge \tau_n} \right) du.
    \end{align*}
To estimate $A^2$ we note that  matrices $yy^\top$ and $a(u,y,c)$
are positive semi-definite,  and hence the condition (LG) implies
\begin{align*}
Tr\left( yy^\top a(u,y,c)\right) \leq Tr\left( yy^\top \right) Tr
\left( a(u,y,c)\right) \leq K |y|^2 (1 + |y|^2).
\end{align*}
Therefore
\begin{align*}
& 2m |y|^{2m-2} Tr(a(u,y,c)) +2m(2m-2)|y|^{2m-4 } Tr(yy^\top  a(u,y,c)) \\
& \leq 2m |y|^{2m-2}K( 1 + |y|^2) + 2m(2m-2)|y|^{2m-4 } K |y|^2 (1 + |y|^2) \\
&\leq 2m(2m-1) K   |y|^{2m-2} ( 1 + |y|^2) \leq  2m(2m-1) K  L ( 1
+ |y|^{2m})
\end{align*}
for some positive $L$.
This gives
\begin{align*}
        \mathbb{E}
        \sup_{ s  \leq t \wedge \tau_n }
        \abs{A^2_s }
        \leq
        L_2
        \mathbb{E}
        \int_0^{t \wedge \tau_n}
        ( 1 + |Y_{u-}|^{2m})
        du
        \leq
        L_2
        \int_0^{t }
        ( 1 + \mathbb{E}Y_{u \wedge \tau_n }^{*,2m})
        du.
\end{align*}
To estimate $M^c$ we use the Burkholder-Davies-Gundy inequality (see, e.g., \cite[Thm. IV.4.48]{prot2004}), the Young inequality with $\varepsilon$ ($|ab| <
\varepsilon a^2 + \frac{1}{ 4 \varepsilon } b^2$) and obtain
\begin{align*}
        & \mathbb{E}
        \sup_{v \leq t   \wedge \tau_n}
        \abs{
        \int_{0}^v
        |Y_{v-}|^{2m-2}
        Y_{v-}^{\!\top}
        d Y^c_v
        }
        \leq
        C
        \mathbb{E}
        \sqrt{
        \int_{0}^{t\wedge \tau_n}
        |Y_{v-}|^{4m-4}
        Y_{v-}^{\!\top} a(v,Y_{v-},C_{v-}) Y_{v-}
        dv
        }
        \\
        &
        \leq
        C
        \mathbb{E}
        \sqrt{
        Y^{*,2m}_{t\wedge \tau_n}
        \int_{0}^{t\wedge \tau_n}
        |Y_{v-}|^{2m-2}
        \abs{Tr (a(v,Y_{v-},C_{v-}))}
        dv
        }
        \\
        &
        \leq
        C
        \varepsilon
        \mathbb{E}
        Y^{*,2m}_{t\wedge \tau_n}
        +\frac{C}{4\varepsilon}
        \mathbb{E}
        \int_{0}^{t\wedge \tau_n}
        |Y_{v-}|^{2m-2}
        \abs{Tr (a(v,Y_{v-},C_{v-}))}
        dv
        \\
        &
        \leq
        C
        \varepsilon
        \mathbb{E}
        Y^{*,2m}_{t\wedge \tau_n}
        +\frac{CK'_1}{4\varepsilon}
        \mathbb{E}
        \int_{0}^{t\wedge \tau_n}
        \left(
        1 + \abs{Y_{v-}}^{2m}
        \right)
        dv
        \leq
        C
        \varepsilon
        \mathbb{E}
        Y^{*,2m}_{t\wedge \tau_n}
        +\frac{CK'_1}{4\varepsilon}
        \int_{0}^{t}
        \left(
        1 + \mathbb{E}Y_{v \wedge \tau_n }^{*,2m}
        \right)
        dv.
\end{align*}
 Hence, taking
$\varepsilon = \frac{1}{ 8 C m}$, we have
\begin{align*}
\mathbb{E} \sup_{v \leq t   \wedge \tau_n}
        \abs{ M^c_v }
        \leq
        \frac{1}{4}
        \mathbb{E}Y_{t \wedge \tau_n }^{*,2m}
        +
                L_3 \int_0^{t }
        ( 1 + \mathbb{E}Y_{u \wedge \tau_n }^{*,2m})
        du.
\end{align*}
%
 To estimate $M^{d}$ we use again the Burkholder-Davies-Gundy and Young inequalities, which give
\begin{align*}
&
\mathbb{E} \sup_{v \leq t   \wedge \tau_n}
        \abs{
        \int_{0}^v
        \int_{ \mathbf{R}^{d+1}  }
        \abs{Y_{u-}}^{2m-2}
        Y_{u-}^{\!\top}
        z_1
        \widetilde{\mu}(du, dz)
        }
\leq
C
\mathbb{E}
    \sqrt{
        \int_{0}^{t   \wedge \tau_n}
\!\!
        \int_{ \mathbf{R}^{d+1}  }
        \abs{Y_{u-}}^{4m-2}
        \abs{
        z_1}^2
        \mu(du, dz)
    }
\\
&
\leq
C
\mathbb{E}
    \sqrt{
           Y_{t   \wedge \tau_n}^{*,2m}
        \int_{0}^{t   \wedge \tau_n}
        \int_{\mathbf{R}^{d+1} }
        \abs{Y_{u-}}^{2m-2}
        \abs{
        z_1}^2
        \mu(du, dz)
    }
\\
&
\leq
C
\varepsilon
\mathbb{E}
           Y_{t   \wedge \tau_n}^{*,2m}
+
\frac{C}{4\varepsilon}\mathbb{E}
        \int_{0}^{t   \wedge \tau_n}
        \!\!\!\!\!
        \abs{Y_{u-}}^{2m-2}
        \!\!\!
        \left(
        \int_{ \mathbf{R}^d  }
        \!\!\!\!\!
        \abs{
        F(u, Y_{u-}, C_{u-}, x)}^2
        \nu(dx)
        +
        \!\!\!\!\!\!\!
        \sum_{k \in \mathcal{K} \setminus C_{u-}}
        \!\!\!\!\!\!\!
\abs{
        \rho^{C_{u-}, k}(u,Y_{u-})}^2 \lambda^{C_{u-}, k}(u,Y_{u-})
        \!\!\right)du
\\
&\leq
C
\varepsilon
\mathbb{E}
           Y_{t   \wedge \tau_n}^{*,2m}
+
\frac{CK'}{4\varepsilon}\mathbb{E}
        \int_{0}^{t   \wedge \tau_n}
        \left(
        1 +
        \abs{Y_{u-}}^{2m}
        \right)
du
\leq
C
\varepsilon
\mathbb{E}
           Y_{t   \wedge \tau_n}^{*,2m}
+
\frac{CK'}{4\varepsilon}\mathbb{E}
        \int_{0}^{t}
        \left(
        1 +
        Y_{u   \wedge \tau_n}^{*,2m}
        \right)
du ,
\end{align*}
where the fourth inequality follows from (LG) and \eqref{eq:LG-2}.
Therefore, taking again $\varepsilon = \frac{1}{8C m}$, we obtain
\begin{align*}
\mathbb{E} \sup_{v \leq t   \wedge \tau_n}
        \abs{ M^d_v }
        \leq
        \frac{1}{4}
        \mathbb{E}Y_{t \wedge \tau_n }^{*,2m}
        +
                L_4 \int_0^{t }
        ( 1 + \mathbb{E}Y_{u \wedge \tau_n }^{*,2m})
        du.
\end{align*}
To estimate $D^1$ we use Taylor expansion for $f(y) = y^{2m}$ and we get
\begin{align*}
&\abs{|y + z_1|^{2m} - |y|^{2m} - 2m |y|^{2m-2} y^{\!\top} z_1 } \leq
m(2m-1) |y + \theta_{z_1} z_1 |^{2m-2} |z_1|^2
\\
& \leq K^{iv} \left( |y|^{2m-2} |z_1|^2 + | z_1 |^{2m} \right),
\end{align*}
since $|\theta_{z_1}| \leq 1 $.
Let
\[
    U(y,z_1) := |y|^{2m-2} |z_1|^2 + | z_1 |^{2m}.
\]
Then
\begin{align}
    \label{eq:supD2}
    \sup_{ u \leq  t \wedge \tau_n} |D^1_u |
    \leq
    K^{iv}
    \int_0^{t \wedge \tau_n}
    \!\!\!\!
    \int_{\mathbf{R}^{d+1}}
    U(Y_{u-},z_1)
    \overline{\mu}(du, dz).
\end{align}
The function $U$ is nonnegative, $\overline{\nu}(u,Y_{u-},C_{u-},dz)du$ is the
predictable projection of $\overline{\mu}(du,dz)$ (see \cite[Thm.
II.1.1.21]{js1987}), so by Theorem  II.1.1.8 \cite{js1987}
    \begin{align}\label{eq:GPi=Gnu}
    &\mathbb{E} \!\!\int_0^{t \wedge \tau_n}
    \!\!\!\!
    \int_{ \mathbf{R}^{d+1} }
    \!\!\!\!
    U(Y_{u-},z_1)
    \overline{\mu}(du, dz)
    =
        \mathbb{E} \!\!\int_0^{t \wedge \tau_n}
    \!\!\!\!
    \int_{ \mathbf{R}^{d+1} }
    \!\!\!\!
        U(Y_{u-},z_1)
    \overline{\nu}(u,Y_{u-},C_{u-},dz) du.
    \end{align}
    By \eqref{eq:LG-2m} and \eqref{eq:LG-2}
    \begin{align*}
    \int_{ \mathbf{R}^{d+1} }
    \!\!\!\!
        U(y,z_1)
    \overline{\nu}(u,y,c,dz)
    =
    \int_{ \mathbf{R}^{d} }
    \!\!\!\!
        U(y,F(u,y,c,x))
    \nu(dx)
    +
    \sum_{j \in \mathcal{K}: j \neq c }    \!\!\!\!
        U(y,\rho^{c,j}(u,y)) \lambda^{c,j}(u,y)
    \leq
        K^{iv}(1+ |y|^{2m}).
    \end{align*}
Hence, and by \eqref{eq:supD2} and \eqref{eq:GPi=Gnu}, we have
    \begin{align*}
        \mathbb{E}\sup_{ 0 \leq  u \leq t \wedge \tau_n} |D^1_u |
    \leq
        L_5 \int_0^{t }
    \left( 1 + \mathbb{E} Y_{u \wedge \tau_n}^{*,2m}\right) du.
    \end{align*}
%
    Taking into account \eqref{eq:S2m} and estimates of all summands of $|Y_t|^{2m}$ we obtain \eqref{eq:gronwall-ineq}
    with $K= 2(L_1 + \ldots + L_5)$. The proof is now complete.
\end{proof}
\begin{cor} \label{eq:LG+LG-2}
Condition (LG) and  \eqref{eq:LG-2} imply
\begin{align*}
        \mathbf{E} \left(\sup_{t \in [\![r,T]\!]} |Y_t|^{2} \right)< \infty.
    \end{align*}
\end{cor}
\begin{cor}
Assume \eqref{eq:LG-2}.
    If $\nu \equiv 0$,
    then
    \eqref{eq:sup-S2} holds for every natural number $m$.
\end{cor}
 \begin{cor}
    Assume 
        \begin{align}\label{eq:LG-rho}
    \left|\rho^{c,k}(t, y) \right|^2 \leq K_3(1 + |y|^2),
    \end{align}
    and the following order of linear growth  of $F$ 
    \begin{align}\label{eq:K-LG}
         {|F(t,y,c,x)|} \leq K(x) ( 1 + |y|)
    \end{align}
    for some function $K$. Then the conditions
    \[
        \int_{ \norm{x} \leq a } K^2(x) \nu (dx)  < \infty,
    \]
    and
    \begin{align}\label{eq:K-2m}
    \int_{ \mathbf{R}^n } K^{2m}(x) \nu (dx)  < \infty
    \end{align}
    for some natural $m$, imply \eqref{eq:sup-S2} for that $m$.
\end{cor}
\begin{cor}
    If the measure $\nu$ has a bounded support, \eqref{eq:LG-rho} holds,
and \eqref{eq:K-LG} is satisfied with $K$ being continuous, then
\eqref{eq:sup-S2} holds for every natural number $m$.
\end{cor}
\begin{cor}
If the measure $\nu$ has all moments, \eqref{eq:LG-rho} holds, and
\eqref{eq:K-LG} is satisfied with $K$ having polynomial growth, then
\eqref{eq:sup-S2} holds for every natural number $m$.
\end{cor}

\section{General martingale problem for \cadlag processes}

In this section we consider a martingale problem
connected with the law of components $(Y,C)$ of solution to \eqref{eq:SDE-gen}. The law of $(Y,C)$ is a measure on an appropriate Skorochod space, since $(Y,C)$ is a \cadlag process as a solution to \eqref{eq:SDE-gen}. So we consider the canonical space
$\Omega = D_{[\![0,T]\!]}(\mathbf{R}^d \times \mathcal{K})$,
i.e. the space of c\`{a}dl\`{a}g functions on $[\![0,T]\!]$ with values in $\mathbf{R}^d \times \mathcal{K}$. On $\Omega$ we consider the coordinate process  $\pi_t:(y,c) \mapsto (y_t, c_t)$, $t \in [\![0,T]\!],$ and canonical filtration $\mathbb{F}$ generated by the coordinate process.
Moreover, let $\mathbb{F}^r =  (\mathcal{F}^r_s)_{s\geq r}$, $\mathcal{F}^r = \sigma (\mathcal{F}^r_s : s\geq r)$, where
$\mathcal{F}^r_s =\sigma(\pi_t : t \in [\![r,s]\!])$.


Let us denote by  $C^{2}  =
C^{2}( \mathbf{R}^d  \times \mathcal{K})$   -
the  space of all measurable functions $v :
\mathbf{R}^d \times \mathcal{K} \rightarrow \mathbf{R} $ such that
$ v(\cdot, k ) \in  C^{2}(\mathbf{R}^d
)$  for every $ k \in \mathcal{K} $, and let $C^{2}_c$ be a set
of functions $f \in C^{2}$ with compact support.

\begin{defn}
We say that a law $\mathbb{P}$ on the  space
$(\Omega,\mathcal{F}^r)$
solves a time-dependent (local) martingale
problem  started at time $r$  with initial distribution $\eta$ for a family of operators $\mathcal{A}=(\mathcal{A}_t)_{t \in [\![ r,T ]\!]}$ if

\begin{enumerate}
 \item the measure $\eta$ is a distribution of $(y_r,c_r)$,

  \item for every function $v \in
C^{2}_c$
  the process
\[
    M^v_t := v(y_t, c_t) - \int_r^t \mathcal{A}_u v(y_{u-}, c_{u-})  du
    \]
is an $\mathbb{F}^r$ (local) martingale on $[\![ r,T ]\!]$ with respect to $\mathbb{P}$.
\end{enumerate}
We say that the (local) martingale problem for $\mathcal{A}$ is well-posed if for any $r \in ]\!]0,T [\![$ and $(y,c) \in \mathbf{R}^d \times \mathcal{K}$ there exists exactly one solution to the martingale problem for $A$ started at time $r$ with initial distribution $\delta_{y,c}$ (cf. Stroock \cite{stroock1975}).
\end{defn}
\begin{rem}
In the case of diffusion processes, so in the case when the
operators  $(\mathcal{A}_t)_{t \in [\![ r,T ]\!]}$ are the second order differential operators, a
solution to the local martingale problem for  $(\mathcal{A}_t)_{t \in [\![ r,T ]\!]}$ is also
a solution to the martingale problem for  $(\mathcal{A}_t)_{t \in [\![ r,T ]\!]}$.
This follows from the fact that  for  each $v \in C^{2}_c$ the function $\mathcal{A}_t v $ is bounded and
continuous, which implies that  the local martingale
$M^v$ is a true martingale.
 However, for processes with jumps this is not the case,
  unless you make quite restrictive assumption on coefficients that ensure boundedness of $\mathcal{A}_t v  $ for $v \in C^{2}_c$, see  Kurtz \cite[Thm. 3.1]{kurtz2010}.
\end{rem}
In the next proposition we formulate result which generalize  the well known fact for diffusions (see, e.g. Kallenberg \cite[Thm. 18.10]{kallenberg}) which states that the well-posedness of the martingale
problem for $\mathcal{A}$ implies the existence of solutions to the martingale problem with an arbitrary initial distributions. Results of such type for L\`{e}vy type operators, under different assumptions, can be found e.g. in Stroock \cite{stroock1975} or Ethier and Kurtz \cite[Chapter 4]{ethkur1986}.

\begin{prop}
Assume that on the canonical space the martingale problem for $\mathcal{A}$ is well-posed.
For an arbitrary probability measure $\eta $ on ${\mathbb{R}^{d} \times \mathcal{K}}$ the martingale problem   for $\mathcal{A}$ started at $r \in [\![ 0,T ]\!]$ with the initial distribution $\eta$ has the unique solution given by
\begin{align}\label{eq:P-rmu}
\mathbf{P}_{r}^\eta := \int_{\mathbb{R}^{d} \times \mathcal{K}}  \mathbf{P}_{r,y,c}  \ \eta(dy,dc) ,
\end{align}
where $\mathbf{P}_{r,y,c}  $ is the unique solution to the martingale problem
for $\mathcal{A}$  started  at $r \in [\![0,T]\!]$ with initial distribution $\delta_{(y,c)}$.
\end{prop}
\begin{proof}
We start from the proof that the right hand side of \eqref{eq:P-rmu} is well defined.
Let
\[
\mathcal{P}_M := \set{ \mathbf{P}_{r,y,c}:\ (y,c) \in \mathbf{R}^d \times \mathcal{K} }.
\]
First, we prove that $\mathcal{P}_M$ is a measurable subset of $\mathcal{P}$ - the set of all probability measure on Skorochod space $\Omega$.
Let $\mathcal{D} $ be a countable subset of  $C_c^\infty$ which is dense in  $C_c^\infty$, and
\[
\mathcal{Q}:=\set{ \mathbf{P} \in \mathcal{P} : M^v \textnormal{ is a $\mathbf{P}$ martingale on } [\![r,T]\!] \text{ for each } v \in \mathcal{D} }.
\]
Note that for $\mathbf{P} \in \mathcal{Q}$ the required martingale property is equivalent to the
following countable many relations: for $v \in \mathcal{D}$, $ s < t $ in $\mathbb{Q}\cap [\![r,T]\!] \cup \set{r,T}$, $A \in \mathcal{G}_s $
\[
\mathbf{E}_{\mathbf{P}} (( M^v_t - M^v_s ) \I_{A} ) = 0,
\]
where $\mathcal{G}_s$ is a countable set of generators of $\mathcal{F}_s$.
For fixed $v \in \mathcal{D}$, $ s < t $ in $\mathbb{Q}\cap [\![r,T]\!] \cup \set{r,T}$, and $A \in \mathcal{G}_s $
consider the mapping
\[
	Y_{s,t,v,
A} : \mathcal{P}  \mapsto \mathbf{R} \cup \set{\infty}
\]
defined
by
\[
	Y_{s,t,v,
A}(\mathbf{P} )
=
\left\{
  \begin{array}{ll}
    \mathbf{E}_{\mathbf{P}} (( M^v_t - M^v_s ) \I_{A} ) , & \hbox{if } \mathbf{E}_{\mathbf{P}} ( \left| M^v_t - M^v_s \right| \I_{A} ) < \infty, \\
    \infty, & \hbox{if } \mathbf{E}_{\mathbf{P}} ( \left| M^v_t - M^v_s \right| \I_{A} ) = \infty. \end{array}
\right.
\]
To show measurability of $\mathcal{Q}$  it is sufficient to notice that $Y_{s,t,v,
A}$ is a measurable mapping. Indeed, measurability of $Y_{s,t,v,
A}$ yields that $\mathcal{Q}$  is measurable since $\mathcal{Q}$ can be represented as the following countable intersection \[
 \mathcal{Q} = \bigcap_{\substack{ s ,t : s < t \\ s ,t \in  \mathbb{Q}\cap [\![r,T]\!] \cup \set{r,T} } } \bigcap_{v \in D } \bigcap_{A \in \mathcal{G}_s} Y^{-1}_{s,t,v,A}(0).
\]
Thus, it remains to show measurability of $Y_{s,t,v,
A}$. For $N > 0$ consider a mapping
\[
	Z_N : \mathcal{P}  \mapsto \mathbf{R}
\]
defined by
\[
Z_N(\mathbf{P}) :=
\mathbf{E}_{\mathbf{P}}( ((( M^v_t - M^v_s )  \wedge N ) \vee (-N)  ) \I_{A}).
\]
We note that the mapping $Z_N$ is measurable. Since on the set on which
$\lim_{N} Z_N$ exists and is finite we have
\[
	Y_{s,t,v,
A}(\mathbf{P}) = \lim_{N} Z_N(\mathbf{P}),
\]
which implies measurability of  $Y_{s,t,v, A}$.
Let
\[
\mathcal{R}:=\set{ \mathbf{P} \in \mathcal{P} : \exists (y,c) \in \mathbf{R}^d \times \mathcal{K} \textnormal{ such that } \mathbf{P}(y_r = y, c_r =c)=1 },
\]
so $\mathcal{R}$ is measurable (see Kallenberg \cite[Lem. 1.36]{kallenberg}). We note that
\[
\mathcal{P}_M =\mathcal{Q} \cap \mathcal{R},
\]
thus $\mathcal{P}_M$ is measurable.

In order to prove that $\mathbf{P}_{r}^\eta$ is well defined, we need to prove that the mapping
\[
f:\mathbf{R}^d \times \mathcal{K}\rightarrow \mathcal{P}_M
,
\quad
\textnormal{defined by }
f(y,c) :=\mathbf{P}_{r,y,c}
\]
is measurable. Let
\[
g : \mathcal{P}_M
\rightarrow \mathbf{R}^d \times \mathcal{K}
\]
be defined
by
\[
g(\mathbf{P}_{r,y,c}) = (y,c) .
\]
The well-possedness of martingale problem, and measurability of $\mathcal{P}_M$ imply that $g$
is a measurable bijection, and therefore by theorem of Kuratowski (see e.g. Kallenberg \cite[Theorem A1.7]{kallenberg}) $g$ has measurable inverse which is equal to $f$.

To end the proof of theorem we need
to show that $\mathbf{P}_{r}^\eta$ defined by \eqref{eq:P-rmu}, is the unique solution of the martingale problem for $\mathcal{A}$ for initial distribution $\eta$ started at $r$.
This follows from the fact that the required martingale property is equivalent to the following countably many relations: for $v \in \mathcal{D}$, $ s < t $ in $\mathbb{Q}\cap [\![r,T]\!] \cup \set{r,T}$, $A \in \mathcal{G}_s $
\[
\mathbf{E}_{\mathbf{P}_{r}^\eta} (( M^v_t - M^v_s ) \I_{A} ) = 0,
\]
which, by definition of $\mathbf{P}_{r}^\eta$, can be written in the form
\[
\int_{\mathbb{R}^{d} \times \mathcal{K}}
\mathbf{E}_{r,y,c} (( M^v_t - M^v_s ) \I_{A} )
\eta(dy,dc)= 0.
\]
This equality holds by the well-possedness of martingale problem.
Therefore $\mathbf{P}^\eta_r$ solves the required martingale problem.
Now we consider the issue of uniqueness. Let $\mathbf{Q}^\eta_r$ be  a solution to martingale problem started at $r$ from $\eta$.
For arbitrary
$v \in \mathcal{C}^\infty_c$, $ s < t $ in $[\![r,T]\!]$, $A \in \mathcal{F}_s $ we have
\[
\mathbf{E}_{\mathbf{Q}_{r}^\eta} (( M^v_t - M^v_s ) \I_{A} | y_r,c_r) = 0.
\]
Therefore $\mathbf{Q}^\eta_r( \cdot | y_r,c_r)$ is a solution to martingale problem started at $r$  from distribution $\delta_{(y_r,c_r)}$. This yields (by well possedness) that $\mathbf{Q}^\eta_r( \cdot | y_r,c_r) = \mathbf{P}_{r,y_r,c_r}( \cdot  )$  - $\mathbf{Q}^\eta_r$ a.s., and thus
\begin{align*}
\mathbf{Q}^\eta_r(\cdot) &=
\int_{\mathbf{R}^d \times \mathcal{K}} \mathbf{Q}^\eta_r( \cdot | y_r= y,c_r=c)
\mathbf{Q}^\eta_r(y_r \in dy, c_r \in dc) \\
&=
\int_{\mathbf{R}^d \times \mathcal{K}} \mathbf{P}_{r,y,c}(\cdot)
\mathbf{Q}^\eta_r(y_r \in dy, c_r \in dc)
\\
&=
\int_{\mathbf{R}^d \times \mathcal{K}} \mathbf{P}_{r,y,c\ }(\cdot)
\eta( dy, dc)
=\mathbf{P}^\eta_{r}(\cdot),
\end{align*}
where $\mathbf{Q}^\eta_r( \cdot | y_r= y,c_r=c)$ is a regular version of conditional probability $\mathbf{Q}^\eta_r( \cdot | y_r,c_r)$.
\end{proof}

We can also prove that the  family $\set{\mathbf{P}_{r,y,c} }$ is a Markov family.
\begin{prop}
Assume that on the canonical space the martingale problem for $\mathcal{A}$ is well-posed.
The  family $\set{\mathbf{P}_{r,y,c} : (r,y,c) \in [\![0,T]\!] \times \mathbb{R}^d \times \mathcal{K} }$  is a Markov family.
\end{prop}
\begin{proof}
We have to prove that for arbitrary  $ t \geq r$ and  $A \in \mathcal{F}^t_T=\sigma((y_u,c_u) : u \in [\![t,T]\!])$ we have
\[
\mathbf{P}_{r,y,c}( A | \mathcal{F}^r_t )
=
\mathbf{P}_{t,y_t,c_t}( A  ),
\]
where by $\mathbf{P}_{r,y,c}( \ \cdot \ | \mathcal{F}^r_{t})$  we denote the
regular conditional probability of $\mathbf{P}_{r,y,c}$ with respect to $\mathcal{F}^r_{t} := \sigma((y_u,c_u) : u \in [\![r,t]\!])$.
For every $\omega \in \mathcal{F}^r_{t}$ the probability measure
 $\mathbf{P}_{r,y,c} ( \ \cdot \ |  \mathcal{F}^r_{t})(\omega)$ solves the martingale problem for  $\mathcal{A}$ started at $t$ from  $\delta_{(y_t(\omega),c_t(\omega))}$ (see e.g. Rogers and Williams \cite[Thm. 21.1]{RW2000}). By uniqueness, we obtain
$$\mathbf{P}_{r,y,c} ( A | \mathcal{F}^r_{t})\
=\mathbf{P}_{t,y_t,c_t}(A) \quad \textnormal{for every } A \in \mathcal{F}^t_T , \mathbf{P}_{r,y,c} \ a.s.$$
which implies corresponding Markov family property.
\end{proof}

\section{Uniqueness in law of weak solutions}
%
%

In this section we prove uniqueness of finite-dimensional
distributions of jump-diffusion under some assumptions on coefficients of SDE, and
assumption on intensities. To prove this fact we use the
martingale problem and prove that it is well-posed.
%
As a consequence of Corollary \ref{cor:mart-prob}  we see that the
law of components $(Y,C)$ of a  solution to SDE \eqref{eq:SDE-gen}
solves a time dependent local martingale problem for the family of operators $(\mathcal{A}_t)_{ t \in [\![r,T]\!]}$ (see e.g. \cite[IV.7.A and IV.7.B]{ethkur1986}).
If we prove that a solution to the local martingale problem is also a solution to the martingale problem and the  martingale problem is
well-posed, then these components constitutes a time inhomogeneous
Markov process and uniqueness in law of $(Y,C)$ holds
(see e.g. \cite[Section IV.4 Theorem 4.1 and Theorem 4.2]{ethkur1986} or
Stroock \cite[Theorem 4.3]{stroock1975}).
However, if the law of process $(Y,C)$ solves the local martingale problem for
$\mathcal{A}$ given by \eqref{eq:gener-SDE}, then the law of $(Y,C)$ does not
necessarily  solves the martingale problem for $\mathcal{A}$.
Nevertheless,
\begin{thm}\label{prop:localmp-mp}
Assume that the law of $(Y,C)$ solves the local martingale problem for
$(\mathcal{A}_t)_{t \in [\![r,T]\!]}$. If coefficients of $(\mathcal{A}_t)_{t \in [\![r,T]\!]}$ satisfies (LG),
and \eqref{eq:LG-2} or
\begin{align}\label{eq:unif-bound}
\int_{|y| \leq 1} \norm{F(t,y,c,x)}^2 \nu(dx) \leq M,
\end{align}
then the law of $(Y,C)$ solves the martingale problem for
$(\mathcal{A}_t)_{t \in [\![r,T]\!]}$.
\end{thm}
\begin{proof}
First assume that (LG) and \eqref{eq:unif-bound} holds and suppose that the law of $(Y,C)$ solves local martingale problem for $(\mathcal{A}_t)_{t \in [\![r,T]\!]}$. This implies that
 $\mathcal{A}_t f$ is bounded for $f \in C^{2}_c$. Therefore,
for any $f \in C^{2}_c$, the local martingale $M^f$ is bounded, so it is a martingale. This proves that if the law of $(Y,C)$ solves the local martingale problem for $(\mathcal{A}_t)_{t \in [\![r,T]\!]}$ then it is also a solution to the martingale problem for $(\mathcal{A}_t)_{t \in [\![r,T]\!]}$.

Now assume that (LG) and \eqref{eq:LG+LG-2} hold and suppose that the law of $(Y,C)$ solves the  local martingale problem for $(\mathcal{A}_t)_{t \in [\![r,T]\!]}$.
From  Corollary  \ref{eq:LG+LG-2} we infer that every
solution to the local martingale problem for $(\mathcal{A}_t)_{t \in [\![r,T]\!]}$ satisfies
\[
    \mathbb{E} \left(\sup_{ t  \in [\![ r, T ]\!] } |Y_t|^2 \right)< K
\]
for some constant $K>0$. Moreover, for  $f \in C^{2}_c$ it follows
from \eqref{eq:LG-2} and (LG) that
\[
\left|\mathcal{A}_t f (y,c) \right| \leq K_f (1 + |y|^2).
\]
This implies that the local martingale $M^f$ has an integrable supremum
and therefore is a martingale.
\end{proof}
\begin{rem}
 The set of conditions (LG)  and \eqref{eq:unif-bound}
comparing  to (LG)  and \eqref{eq:LG-2} is on the one hand more  restrictive and on the other hand less restrictive because there is not assumptions on $\rho$.
\end{rem}


\begin{thm}\label{thm:markov-mp}
Let continuous functions $\lambda^{i,j}$, $i \neq j,\  i,j \in
\mathcal{K}$, satisfy $(\Lambda)$ and
\begin{flalign*}
(\Lambda_0) && \lambda^{i,j}(t, y) \equiv 0 \qquad \text{ or }
\qquad 0 <  \lambda^{i,j}(t, y)  \quad  \forall t \in [\![0,T]\!] \quad
\forall s \in \mathbf{R}^d. &&
\end{flalign*}
Moreover assume (Lip), (LG), (Cont), (LB) and \eqref{eq:LG-2}.
Then the martingale problem for $(\mathcal{A}_t)_{t \in [\![0,T]\!]} $, defined by
\eqref{eq:gener-SDE}, is well-posed.
\end{thm}
\begin{proof}
Fix arbitrary $r \in [\![0,T[\![$, $(y,c) \in \mathbf{R}^d \times \mathcal{K}$.
Let $\mathcal{ A }^b_t $ be an operator defined by
\begin{align}
\label{eq:gener-SDE-pom} & \mathcal{ A }^b_t v(y,c) :=
\nabla v(t,y,c)\mu(t,y,c) +
\frac{1}{2}
 \Tr \left( a(t,y,c)\nabla^2 v(t,y,c) \right)
\\
\nonumber & + \int_{\mathbf{R}^n}\!\!\!\left( v( t, y  + F(t,y,c,x)
,c) - v( t,y,c) - \nabla v(t,y,c)F(t,y,c,x) \I_\set{ \norm{x} \leq
a } \right)\nu(dx)
\\
\nonumber
 &
+ \sum_{k \in \mathcal{K} \setminus c } \left( v( t, y +
\rho^{c,k}(t, y),k) - v(t,y,c) \right)b^{c,k},
\end{align}
where we put $b^{c,k}\!\!\! :=\! \sup_{(t, y) \in [\![0,T]\!]\times \mathbf{R}^d } \lambda^{c,k}(t, y)$.
The operator $\mathcal{ A }^b_t$ is well defined on $C^{1,2}_c$.
Denote $\mathcal{ A }^b = (\mathcal{ A }^b_t)_{t \in [\![ r,T ]\!]} $.
First, we consider the martingale problem for $\mathcal{ A
}^b$.
Take a  filtered probability space $(\Omega,\mathcal{F},\mathbb{F},
\mathbf{P}^b)$
on which there exist independent processes:
a standard Brownian motion $W$, a  Poisson random measure $\Pi(dx, dt)$  with intensity measure $\nu(dx)dt$
 and the Poisson processes $(N^{i,j})_{i,j \in
\mathcal{K} : i \neq j }$  with intensities equal to $(b^{i,j})_{i,j
\in \mathcal{K} : i \neq j }$ for $b^{i,j} > 0$. If $b^{i,j}=0$,
then  we put $N^{i,j} \equiv 0$. On this stochastic basis we can
construct for arbitrary $(y,c)$, under assumptions of theorem, a unique solution to the
following SDE on $[\![r,T]\!]$:
\begin{align}
\nonumber d Y_t &= \mu( t, Y_{t-} , C_{t-}) dt  + \sigma( t,
Y_{t- }, C_{t-}) d W_t +
\int_{\norm{x} \leq a  } \!\!\!\!F( t, Y_{t- }, C_{t-}, x) \widetilde{ \Pi } (dx , dt)
\\
\label{eq:SDE-pom} &
+
\int_{\norm{x} \geq  a  } \!\!\!\!F( t, Y_{t- }, C_{t-}, x) \Pi  (dx , dt)
  + \sum_{i,j:j\neq i} \rho^{i,j}(t,Y_{t-}) \I_\set{i}(C_{t-})
 d N^{i,j}_t,
\\
\nonumber
 d C_t &= \sum_{i,j \neq i} (j - i) \I_{i} (C_{t-}) d
 N^{i,j}_t ,
\\
\nonumber  Y_r &= y,  \qquad C_r =c.
\end{align}
By Proposition \ref{prop:mart-prob} and Proposition
\ref{prop:localmp-mp} the law of $(Y,C)$ solves the martingale
problem for $\mathcal{A}^b$, since SDE \eqref{eq:SDE-pom} is a
special case of \eqref{eq:SDE-gen}.
The uniqueness of a strong solution to \eqref{eq:SDE-pom} implies pathwise uniqueness and
therefore martingale problem for $\mathcal{A}_b$ is well-posed
\cite[Corollary 140]{rong1997}.


 Let $\overline{\mathbf{P}}^b_{r,y,c}$ denote the law of
this unique solution on $\overline{\Omega} = D_{[\![0,T]\!]}(\mathbf{R}^d)
\times D_{[\![0,T]\!]}(\mathcal{K})$ endowed with $\overline{\mathbf{P}}^b_{r,y,c}$ - completed canonical filtration
$\overline{\mathbb{F}} = (\overline{\mathcal{F}}_t)_{t \in [\![0,T]\!]}$
 generated by the coordinate process $\pi_t : (y,c) \mapsto
(y_t,c_t)$.
Therefore, for every $v \in C^{2}_c$, the process
\begin{align}\label{eq:mart-Mv}
M^{v}_t := v(y_t,c_t)  - \int_{r}^t \mathcal{A}^b_v v( y_{u-}, c_{u-} )  du
\end{align}
is an $\overline{\mathbb{F}}$--martingale under
$\overline{\mathbf{P}}^b_{r,y,c}$. We now introduce on $\overline{\Omega}$
the nonanticipating processes:
\[
    H^{i,j}_t(y,c) := \sum_{r < s \leq t} \I_\set{c_{s-} =i } \I_\set{c_{s} =j },
    \qquad
    H^{i}_t(y,c)  := \I_\set{c_{t} =i }.
\]
By definition, they depend only on the path of process $c$.
 The process $H^{i,j}$ is a
counting process with the $\overline{\mathbb{F}}$--intensity given
by  $\int_r^{\cdot} H^i_{u-} b^{i,j} du  $, i.e., the process
\begin{align}\label{eq:mart-Mij-b}
    M^{i,j}_t :=  H^{i,j}_t - \int_r^{t} H^i_{u-} b^{i,j} du
\end{align}
is an $\overline{\mathbb{F}}$--martingale. Now we define a new
probability measure $\overline{\mathbf{P}}^\lambda_{r,y,c}$ on
$(\overline{\Omega}, \overline{\mathcal{F}}_T )$ by the formula
\begin{equation}\label{eq:density-p-lambda-b}
\frac{d \overline{\mathbf{P}}^\lambda_{r,y,c}}{ d \overline{\mathbf{P}}^b_{r,y,c}}
:= Z_T := \mathcal{E}_T \left( \sum_{(i,j) : b^{i,j} > 0}
\int_{]\!]r, \cdot]\!]} \left( \frac{\lambda^{i,j}(t,y_{t-})}{b^{i,j}} -
1 \right) ( d H^{i,j}_t - H^{i}_{t-}b^{i,j}dt )\right) .
\end{equation}
By our assumptions, this is the probability measure equivalent to
$\overline{\mathbf{P}}^b_{r,y,c}$.
We prove that $\overline{\mathbf{P}}^\lambda_{r,y,c}$ solves our original
martingale problem. To do this  we check that for every $v \in
C^{2}_c$  the process
 \[
  \widehat{M}^v_t := v(y_t,c_t) - \int_{r}^t \mathcal{A}_u v (y_{u-},c_{u-}) du
 \]
 is an $\overline{\mathbb{F}}$--martingale under
 $\overline{\mathbf{P}}^\lambda_{r,y,c}$,
 or equivalently that $\widehat{M}^v
Z$ is an $\overline{\mathbb{F}}$--martingale under
$\overline{\mathbf{P}}^b_{r,y,c}$, where $Z_t :=
\mathbb{E}_{\overline{\mathbf{P}}^b_{r,y,c}} ( Z_T |
\overline{\mathcal{F}}_t )$. To prove this, we first decompose
$\mathcal{A} $ as
\[
\mathcal{A} = \mathcal{A}_b + \mathcal{B},
\]
where $\mathcal{A}_b$ is defined by \eqref{eq:gener-SDE-pom} and
\begin{equation}  \label{eq3/15}
\mathcal{B} v(t,y,c)  := \sum_{k \neq c } \left( v( t, y +
\rho^{c,k}(t, y),k) - v(t,y,c) \right)(\lambda^{c,k}(t, y) -
b^{c,k}).
\end{equation}
 Obviously,
\[
    \widehat{M}^v_t = M^v_t - \int_r^t \mathcal{B} v(u,y_{u-},c_{u-}) du,
\]
where $M^v$ is given by \eqref{eq:mart-Mv}, so $M^v$ is an
$\overline{\mathbb{F}}$--martingale under
 $\overline{\mathbf{P}}^b_{r,y,c}$.  Therefore, by integration by parts formula, we have
\begin{align}\label{eq:dMvtetat}
    d
    (\widehat{M}^v_t Z_t)
    =
     \widehat{M}^v_{t-} d Z_t + Z_{t-} d  M^v_t  - Z_{t- } \mathcal{B }v(t-,y_{t-},c_{t-}) dt
    +
    d [ Z, \widehat{M}^v ]_t.
\end{align}
Since $\left( \left[Z, \widehat{M}^v \right]_t  - \left\langle Z, \widehat{M}^v \right\rangle_t \right)_{ t \in [\![ 0,T ]\!]}$ is a local martingale to obtain that $\widehat{M}^v Z$ is a $\overline{\mathbf{P}}^b_{r,y,c}$  local martingale it is enough to  show that
\begin{equation}  \label{eq1/15}
    d  \left\langle Z, \widehat{M}^v \right\rangle_t =
    Z_{t- } \mathcal{B }v(t-,y_{t-},c_{t-} )  dt.
\end{equation}
Let us denote by
$\overline{\mu}$ the measure of jumps of $(y,c)$.
By the Theorem \ref{thm:semimart-dec} we find that the
process $(y,c)$ is a semimartingale under $\mathbb{P}^b$, with the
compensator of measure of jumps of $(y,c)$, denoted by
$\overline{\nu}$, given by \eqref{eq:komp-sc}.
Note that we have
\begin{align*}
d [ Z, & \widehat{M}^v ]_t  =  \Delta Z_t \Delta
\widehat{M}^v_t =\\& \sum_{\substack{ i,j \in \mathcal{K}, \\ j\neq
i , b^{i,j}
> 0 }} Z_{t-} \left( \frac{\lambda^{i,j}(t,y_{t-})}{b^{i,j}} - 1
\right) \left( v( t, y_{t-} \!+ \Delta y_t,i + \Delta c_t) -
v(t,y_{t-},i) \right) \Delta H^{i,j}_t  .
\end{align*}
Since, for
$j \neq i$,
\[
\Delta H^{i,j}_t
=H^i_{t-} \I_\set{ \Delta
c_{t} = j-i}
\]
we have
\begin{align*}
&\left( v( t, y_{t-} \!+ \Delta y_t,i + \Delta c_t) -
v(t,y_{t-},i)
\right) \Delta H^{i,j}_t
\\
    & =\int_{\mathbf{R}^{d+1} }
\left( v( t, y_{t-} \!+ z_1,i + z_2) - v(t,y_{t-},i) \right)
H^i_{t-} \I_\set{ z_2 = j-i} \overline{\mu}( \omega, \set{t},dz_1,
dz_2  ).
\end{align*}
Now we will show that
\begin{align} \label{eq2/15}
&\int_{\mathbf{R}^{d+1} }
\!\!\!\!\!\!\!\!\!\!\left( v( t, y_{t-} \!+ z_1,i + z_2) -
v(t,y_{t-},i) \right) H^i_{t-} \I_\set{ z_2 = j-i} \widetilde{\nu}(
dt ,dz_1, dz_2  ) \\
& \notag = \left( v( t, y_{t-} \!+ \rho^{i,j}(t,y_{t-}),j) -
v(t,y_{t-},i) \right) H^{i}_{t-} b^{i,j} dt.
\end{align}
Using formulae \eqref{eq:komp-sc-exact} and \eqref{eq:komp-sc} defining the compensator $\widetilde{\nu}$ and its intensity $\overline{\nu}$ we
see that we  have to  compute two integrals. The first integral is
equal to
\begin{align}
\label{eq:d-eta-Mv}
&\int_{\mathbf{R}^{d+1}}
\!\!\!\!\!\!\!\!\!\!\left( v( t, y_{t-} \!+ z_1,i + z_2) -
v(t,y_{t-},i) \right) H^i_{t-} \I_\set{ j-i}(z_2) \nu_F(t,y_{t-},c_{t-}, d
y_1) \otimes \delta_\set{0}(dz_2)
\\
\nonumber
& =\int_{\mathbf{R}^d} \left( \int_{ \mathbf{R}}
\!\!\!\!\!\left( v( t, y_{t-} \!+ z_1,i + z_2) -
v(t,y_{t-},i) \right) H^i_{t-} \I_\set{j-i}(z_2)
\delta_\set{0}(dz_2) \right)\nu_F(t,y_{t-},c_{t-},dz_1)
\\
\nonumber
&= \int_{\mathbf{R}^d} \left( v( t, y_{t-} \!+ z_1,i ) -
v(t,y_{t-},i) \right) H^i_{t-} \I_\set{ j-i}(0)
 \nu_F(t,y_{t-},c_{t-},dz_1) = 0,
\end{align}
and the second integral is equal to
\begin{align*}
&\int_{\mathbf{R}^{d+1}}
\!\!\!\!\!\!\!\!\!\!\left( v( t, y_{t-} \!+ z_1,i + z_2) -
v(t,y_{t-},i) \right) H^i_{t-} \I_\set{ z_2 = j-i} \left(\sum_{ k
\in \mathcal{K} \setminus \set{c_{t-}}} b^{c_{t-},k} \delta_\set{
(\rho^{c_{t-},k}(t,y_{t-}) , k-c_{t-})} (dz_1, dz_2) \right)
\\
&=\int_{\mathbf{R}^{d+1}}
\!\!\!\!\!\!\!\!\!\!\left( v( t, y_{t-} \!+ z_1,i + z_2) -
v(t,y_{t-},i) \right) H^i_{t-} \I_\set{ z_2 = j-i} \left(\sum_{ k
\in \mathcal{K} \setminus \set{i}} b^{i,k} \delta_\set{
(\rho^{i,k}(t,y_{t-}) , k-i)} (dz_1, dz_2) \right)
\\
&= \sum_{ k \in \mathcal{K} \setminus \set{i}} \left( v( t, y_{t-} \!+
\rho^{i,k}(t,y_{t-}),k) - v(t,y_{t-},i) \right) H^i_{t-}
\I_\set{
k = j}  b^{i,k} \\
&=
 \left( v( t, y_{t-} \!+ \rho^{i,j}(t,y_{t-}),j) - v(t,y_{t-},i) \right) H^i_{t-}
 b^{i,j},
\end{align*}
so \eqref{eq2/15} holds.
 Since
\[
d [ Z, \widehat{M}^v ]_t
=
\!\!\!
\sum_{ i,j: j \neq i }
\!\!
Z_{t-}
\!\!
\left(\!\frac{\lambda^{i,j}(t,y_{t-}) }{
b^{i,j} } - 1 \!\right)
\!\!
\left( v( t, y_{t-} \!+ \rho^{i,j}(t,y_{t-}),j) -
v(t,y_{t-},i) \right) d H^{i,j}_t(c)
\]
and, by \eqref{eq3/15},
\[
Z_{t- } \mathcal{B }v(t,y_{t-},c_{t-})
=
\!\!\!
\sum_{ i,j: j \neq i }
\!\!
Z_{t-}
\!\!
\left(\!\frac{\lambda^{i,j}(t,y_{t-}) }{
b^{i,j} } - 1 \!\right)
\!\!
\left( v( t, y_{t-} \!+ \rho^{i,j}(t,y_{t-}),j) -
v(t,y_{t-},i) \right) H^{i}_{t-}(c) b^{i,j},
\]
using \eqref{eq:d-eta-Mv}, \eqref{eq2/15} and \eqref{eq:mart-Mij-b}
we get from \eqref{eq:dMvtetat}  that
\begin{align*}
    &d
    (\widehat{M}^v_t Z_t)
    =
    \widehat{M}^v_{t-} d Z_t   + Z_{t-} d  M^v_t
    \\
    &
    +
    \sum_{ i,j: j \neq i } Z_{t-}
    \left(\!\frac{\lambda^{i,j}(t,y_{t-}) }{
    b^{i,j} } - 1 \!\right)
    \!
    \left( v( t, y_{t-} + \rho^{i,j}(t,y_{t-}),j) -
    v(t,y_{t-},i) \right)d M^{i,j}_t.
\end{align*}
Hence $\widehat{M}^v$ is a
$\overline{\mathbf{P}}^\lambda_{r,y,c}$ local martingale. By assumption (LB)
and Proposition \ref{prop:localmp-mp},  $\widehat{M}^v$ is a true
martingale under $\overline{\mathbf{P}}^\lambda_{r,y,c}$. Similar
argumentation shows that if $\mathbf{Q}^\lambda_{r,y,c}$ solves the
martingale problem for $\mathcal{A}$, then $d \mathbf{Q}^b_{r,y,c} :=
Z_T^{-1} d \mathbf{Q}^\lambda_{r,y,c}$, where $Z_T$ is given by
\eqref{eq:density-p-lambda-b}, solves the martingale problem for
$\mathcal{A}^b$. This implies that the martingale problem for
$\mathcal{A}$ is well-posed. Contrary, let
$\overline{\mathbf{P}}^\lambda_1$, $\overline{\mathbf{P}}^\lambda_2$
be two solutions of the martingale problem for $\mathcal{A}$ and
$\overline{\mathbf{P}}^\lambda_1 \neq
\overline{\mathbf{P}}^\lambda_2$. Then $\overline{\mathbf{P}}^b_i$
given by $d \overline{\mathbf{P}}^b_i = Z_T^{-1} d
\overline{\mathbf{P}}^\lambda_i$ for $i = 1,2$ solve the martingale
problem for $\mathcal{A}_b$ and $\overline{\mathbf{P}}^b_1 \neq
\overline{\mathbf{P}}^b_2$. This is a contradiction with
well-posedness of the martingale problem for $\mathcal{A}^b$.
\end{proof}

\section{ Examples }
Now, we present two examples illustrating how our results work: a generalized exponential Levy model and a semi-Markovian regime switching model.
\begin{example}[Generalized exponential L\'evy models]
This model generalize exponential Levy model described, e.g., in \cite{convol2005b}. Consider the following SDE
\begin{align*}\label{eq:gen-levy-exp}
    d Y_t &= Y_{t-}\!\!\left( \sigma(C_{t-}) d W_t  + \!\!\int_{\mathbf{R}}\!\!
    (e^{\sigma(C_{t-}) x } - 1 ) \widetilde{\Pi}( dx, dt)+ \!\!\!\!\sum_{i,j \in
    \mathcal{K}: j \neq i } \!\!\!\!(e^{\rho^{i,j}} - 1 )  H^{i}_{t-}d N^{i,j}_t\right) \\
    d C_t &= \sum_{i,j \in \mathcal{K} : j \neq i} (j-i)\I_\set{ i} (C_{t-}) d
    N^{i,j}_t ,
\end{align*}
where $\sigma(i) \geq 0$, $\rho^{i,j}
\in \mathbf{R}$, $N^{i,j}$ are independent Poisson processes with
constant intensities $\lambda^{i,j} >0 $, $\Pi(dx,dt)$ is a
Poisson random measure with intensity measure $\nu(dx)dt$ satisfying
\[
    \int_{|x|>1}e^{2 \sigma(i) x} \nu(dx) <
    \infty
    \qquad
    \forall
    i \in \mathcal{K}.
\]
Note that the coefficients of this SDE satisfy assumptions of Theorem
\ref{thm:markov-mp}, so there exists a solution unique in law. Moreover, by
using the Ito lemma one can show that this unique solution is of the
  form:
\[
    Y_t = Y_0 \exp\left( \int_{0}^t J(\sigma(C_{u-})) du + \int_{0}^t\sigma( C_{u-}) d Z_u  + \sum_{i,j\in \mathcal{K}: j \neq i } \int_0^t  \rho^{i,j} \I_\set{ i }(C_{u-})d N^{i,j}_u\right),
\]
where $Z$ is a L\'evy  process with the Levy-Ito decomposition:
\[
Z_t = W_t + \int_{0}^t \int_{|x| \leq 1} x \widetilde{\Pi}(dx, du) + \int_{0}^t \int_{|x| > 1} x \Pi(dx, du)
\]
and
\[
    J(u) := - \frac{u^2}{2} + \int_{\mathbf{R}} e^{u x } - 1 + ux \I_\set{|x| < 1} \nu(dx).
\]
Moreover, the  coordinate $C$ of the solution $(Y,C)$ is a Markov chain
with the state space $\mathcal{K}$.
\end{example}

\begin{example}[Semi-Markovian regime switching models]
In this example we will illustrate how a feed-back mechanism in jumps of $Y$ and intensity of jumps of $C$ give in our framework extra flexibility in modelling. We present how semi-Markov switching processes can be embedded in our framework. Let us recall that semi-Markov nature of $C$ is reflected in  the fact that the compensator of jumps from $i$ to $j$, $\lambda_{i,j}$ depends on time that process $C$ spends in current state after the last jump. We recall basic facts from theory of semi-Markov processes. The  semi-Markov process $C$ is related with a pair $(X,T)=(X_n, T_n: n \geq 1)$
which is a homogenous Markov renewal process, i.e. 
 \[
	\mathbb{P}( X_{n+1} = j , T_{n+1} - T_{n} \leq t | X_0, \ldots, X_{n} ; T_0, \ldots, T_{n} )
=
\mathbb{P}( X_{n+1} , T_{n+1} - T_{n} \leq t | X_{n}   ) = Q_{X_{n},j} (t)
\]
for every $t\geq 0$, $j \in \mathcal{K}$. $Q$ is called a semi-Markov kernel.
Let
\[
	P_{i,j} := \lim_{t \rightarrow \infty} Q_{i,j}(t).
\]
It can be shown that $(X_n)$ is a homogenous Markov chain with one-step transition matrix $P=(P_{i,j})$.
A semi-Markov process $C$ is  defined by
\[
	C_t := X_{N_t},
\]
where
\[
	N_t := \sup \set{n : T_n \leq t }.
\]
In general, a semi-Markov process does not have Markov property, the Markov property holds only at times $(T_n)$. If we assume that the distribution of holding times, i.e.
\[
F_{i,j}(t)   := \frac{Q_{i,j}(t)}{P_{i,j}}
\]
has a density $f_{i,j}$, then the related semi-Markov process considered as an MPP has intensity.
It was shown in \cite{hunt2010} that for semi-Markov processes the intensity of jumps from $i$ to $j$ depends on time that process $C$ spends after the last jump in a current state, and has the form
\[
	\lambda_{i,j}(\omega, t)
=
\lambda_{i,j}(R_{t-}(\omega))
=\frac{P_{i,j} f_{i,j}(R_{t-}(\omega)) }{1 - \sum_{m }Q_{i,m}(R_{t-}(\omega))} ,
\]
where
\begin{align}\label{eq:czas}
	R_{t} := t - T_{N_{t}}.
\end{align}
A semi-Markov regime switching process can be embedded in our framework by  considering the following SDE
\begin{align*}\label{eq:gen-levy-exp}
    d S_t &= S_{t-}\!\!\left( r dt + \sigma(C_{t-}) d W_t  + \!\!\int_{\mathbf{R}}\!\!
    (e^{\sigma(C_{t-}) x } - 1 ) \widetilde{\Pi}( dx, dt)\right) \\
    d R_t&= dt - \sum_{i,j \in \mathcal{K} : j \neq i} R_{t-} \I_\set{ i} (C_{t-}) d
    N^{i,j}_t \\
    d C_t &= \sum_{i,j \in \mathcal{K} : j \neq i} (j-i) \I_\set{ i} (C_{t-}) d
    N^{i,j}_t ,
\end{align*}
where $W$ is a Wiener process , $N^{i,j}$ are the point processes with
intensity functions given by
\[
	\lambda^{i,j} (z):=\frac{P_{i,j} f_{i,j}(z) }{1 - \sum_{m }Q_{i,m}(z)},
\] $\Pi(dx,dt)$ is a
Poisson random measure with an intensity measure $\nu(dx)dt$ satisfying
\begin{align}\label{eq:SM-levy-exp}
    \int_{|x|>1}e^{2 \sigma(i) x} \nu(dx) <
    \infty
    \qquad
    \forall
    i \in \mathcal{K}.
\end{align}
and $\sigma(i) \geq 0$. Note that the coefficients of this SDE satisfy assumptions of Theorem \ref{thm:markov-mp}. Moreover, if functions $\lambda^{i,j}$ satisfy assumptions of this theorem,
 then there exists a solution unique in law.
Since the process
\[
\int_0^t \sum_{i,j \in \mathcal{K} : j \neq i} \I_\set{ i} (C_{u-}) d
    N^{i,j}_u
\]
counts the number of jumps of $C$, it is easy to see that the component $(R_t)$ represents process given by \eqref{eq:czas}.
Moreover,
using the It\^{o} lemma one can see that the process $S$  is of the
  form:
\[
    S_t = S_0 \exp\left( \int_{0}^t r + J(\sigma(C_{u-})) du + \int_{0}^t\sigma( C_{u-}) d Z_u  \right),
\]
where $Z$ is a L\'evy  process with the L\'evy-It\^{o} decomposition:
\[
Z_t = W_t + \int_{0}^t \int_{|x| \leq 1} x \widetilde{\Pi}(dx, du) + \int_{0}^t \int_{|x| > 1} x \Pi(dx, du)
\]
and
\[
    J(u) := - \frac{u^2}{2} + \int_{\mathbf{R}} ( e^{u x } - 1 + ux \I_\set{|x| < 1}) \nu(dx).
\]
Moreover, the  coordinate $C$ of the solution $(S,R,C)$ is a semi-Markov chain
with the state space $\mathcal{K}$.
\end{example}

\section{Appendix. Proof of Theorem \ref{thm:ito} }
\begin{proof}
For the brevity, we use the following notation $ \mu^i_t
:= \mu(t, Y_{t-},i)$, $ \sigma^i_t
:= \sigma(t, Y_{t-},i)$, $F^i_t(x) := F(t,Y_{t-},i,x)$,
$\rho^{i,j}_t := \rho^{i,j} (t,Y_{t-})$, $\lambda^{i,j}_t:=
\lambda^{i,j}(t,Y_{t-})$, $v^i(t,x) := v(t,x,i)$. Let us recall that
$H^i$, $H^{i,j}$ are given by \eqref{eq:Hi}, \eqref{eq:Hij}. By
integration by parts formula
\[
d v(t , Y_t , C_t) =  \sum_{i} d (v(t , Y_t , i) H^{i}_t) = I_1 +
I_2 + I_3 ,
\]
where: 
\[
I_1 := \sum_{i}  H^{i}_{t-} d v(t , Y_{t} , i),
\ \ \ \ \ \ \ \
I_2 := \sum_{i}  v^i(t , Y_{t-}) d H^{i}_t,
\ \ \ \ \ \ \ \
I_3 := \sum_{i}  \Delta v^i(t , Y_{t} ) \Delta  H^i_t,
\]
and 
$\Delta v^i(t,Y_t) : = v^i(t , Y_{t} ) - v^i(t , Y_{t-} )$. We start
from calculation of $I_2$.
\begin{align*}
I_2 &= \sum_{i}  v^i(t , Y_{t-} ) d \left(\sum_{ j: j \neq i }
(H^{j,i} - H^{i,j})_t \right) = \sum_{i, j: j \neq i }  v^i(t ,
Y_{t-} ) d H^{j,i}_t  - \sum_{i, j : j \neq i }  v^i(t , Y_{t-} )
d H^{i,j}_t
\\
&= \sum_{i, j :j \neq i }  v^j(t , Y_{t-} ) d H^{i,j}_t  -
\sum_{i, j  :j \neq i }  v^i(t , Y_{t-}) d H^{i,j}_t = \sum_{i, j
:j \neq i }  ( v^j(t , Y_{t-} )   - v^i(t , Y_{t-} ) )d H^{i,j}_t.
\end{align*}
Next, from  the fact that $H^{i,j}$ and $\Pi$ have no common jumps
and the   form of $Y$ (see \eqref{eq:SDE-gen}) we note that
\[
\Delta Y_t \Delta H^{i,j}_t = \rho^{i,j}_t  \Delta H^{i,j}_t.
\]
Therefore, by  the same arguments as in $I_2$, we obtain
\begin{align*}
I_3 &= \sum_{i,j :j\neq i} ( \Delta v^j(t,Y_t) - \Delta v^i(t,Y_t))
\Delta H^{i,j}_t
\\
&= \sum_{i,j :j\neq i} \left( v^j(t,Y_{t-} + \rho^{i,j}_t) -
v^j(t,Y_{t-})- v^i(t,Y_{t-}  + \rho^{i,j}_t)+ v^i(t,Y_{t-})
\right) \Delta H^{i,j}_t .
\end{align*}
Hence 
\begin{align}\label{eq:I2+I3}
I_2 + I_3 = \sum_{i,j:j\neq i} ( v^j(t,Y_{t-}  + {\rho^{i,j}_t} ) -
v^i(t,Y_{t-}  + \rho^{i,j}_t) ) d H^{i,j}_t .
\end{align}
Now we will deal with the first term $I_1$. By the It\^o lemma we
have
\begin{align*}
d v^i(t , Y_{t}) &=  \partial_t v^i(t , Y_{t})dt + \nabla v^i(t,
Y_{t-})  d Y_t +  \frac{1}{2} \Tr  \left( \sigma^i_t
(\sigma^i_t)^\top \nabla^2 v^i (t,Y_{t-}) \right) dt
\\
& + \Delta v^i(t , Y_{t}) - \nabla v^i(t, Y_{t-})\Delta Y_t ,
\end{align*}
where we  use the fact that quadratic variation of continuous
martingale part of $Y^j$ and $Y^k$ is equal to
\begin{equation*}
d [Y^j, Y^k ]^c_t = \sum_{i} H^{i}_{t- } \sum_{l=1}^n [\sigma^i_t
 ]_{j,l}[\sigma^i_t ]_{k,l}dt = \sum_{i} H^{i}_{t- }
[\sigma^i_t (\sigma^i_t)^\top]_{j,k}dt.
\end{equation*}
 The jump $\Delta v^i(t , Y_{t})$ is equal to
\begin{align}
\nonumber H^{i}_{t-} \Delta v^i(t , Y_{t}) &= H^{i}_{t-} (v^i(t ,
Y_{t}) - v^i(t , Y_{t-}))
\\
\label{eq:jump-vi} &= H^{i}_{t-} \int_{\mathbf{R}^n} (v^i(t ,
Y_{t-}  + F^i_t(x) ) - v^i(t , Y_{t-})) \Pi (dx,dt)
\\
\nonumber & + H^{i}_{t-} \sum_{j \neq i } (v^i(t , Y_{t-} +
{\rho^{i,j}_t} ) - v^i(t , Y_{t-})) \Delta H^{i,j}_t.
\end{align}
Moreover, we have
\begin{align}
\label{eq:pvi-jump-s} H^i_{t-}\! \nabla v^i(t, Y_{t-}) \Delta Y_t
\!=\! H^i_{t-} \nabla v^i(t, Y_{t-}) \!\! \left( \!
\int_{\mathbf{R}^n} \!\!\!F^i_t(x) \Pi (dx,dt) \!+\! \sum_{j \neq
i } \rho^{i,j}_t \Delta H^{i,j}_t \!\right)\!.
\end{align}
Using \eqref{eq:jump-vi}, \eqref{eq:pvi-jump-s} and
\eqref{eq:SDE-gen} we get that single summand of $I_1$ is given by
\begin{align*}
& H^i_{t- } d  v^i(t, Y_{t-})) = H^i_{t- } \left(
\partial_t v^i(t, Y_t ) dt
+ \frac{1}{2} \Tr \left( \sigma^i_t  (\sigma^i_t)^\top \nabla^2
v^i(t, Y_{t-}) \right) \right) dt
\\
& + \ H^i_{t- } \nabla v^i(t, Y_{t-}) \!\!\left( \mu^i_t dt +
\sigma^i_t d W_t
+ \int_{\norm{x} <a} \!\!\!\!\!\!F^i_t(x)  \widetilde{\Pi} (dx,dt)
+ \int_{\norm{x} > a} \!\!\!\!\!\! F^i_t(x)  \Pi (dx,dt) + \sum_{j
\neq i } \rho^{i,j}_t d H^{i,j}_t \right)
\\
& +  H^{i}_{t-} \int_{R^n} (v^i(t , Y_{t-} + F^i_t (x) ) - v^i(t
, Y_{t-}) - \nabla v^i(t, Y_{t-})
 F^i_t(x)  ) \Pi (dx,dt) \\
 & +  H^{i}_{t-} \sum_{j \neq i } (v^i(t
, Y_{t-} + {\rho^{i,j}_t}) - v^i(t , Y_{t-}) - \nabla v^i(t,
Y_{t-}) \rho^{i,j}_t  ) d H^{i,j}_t.
\end{align*}
In the above expression the jump part can be
compensated, since $v$ satisfies (by assumption) integrability condition
\eqref{eq:ito-formula-int-cond}.  Therefore the above expression can be rearranged in the following way.
\begin{align*}
H^i_{t- } &d v^i(t,Y_t) \!= \!H^i_{t- } \!\!\Bigg[ \!\!\left(
\partial_t v^i(t, Y_t )
+
 \nabla v^i(t, Y_{t-}) \mu^i_t
+ \frac{1}{2} \Tr \left( \sigma^i_t (\sigma^i_t)^\mtop \nabla^2
v^i(t, Y_{t-}) \right) \!\!\right) \!dt
\\
& + \int_{\mathbf{R}^n}\left( v^i(t , Y_{t-} \!+\! F^i_t(x) ) -
v^i(t , Y_{t-}) - \nabla v^i(t, Y_{t-}) F^i_t (x) \I_{\set{
\norm{x}  < a }} \right) \nu(dx) dt
\\
&+ \sum_{j:j \neq i } \left( v^i(t , Y_{t-}  +
{\rho^{i,j}_t}) - v^i(t , Y_{t-})
\right)\lambda^{i,j}_t dt\Bigg]
\\
&+ \ H^i_{t- } \nabla v^i(t, Y_{t-}) \sigma^i_t d W_t
\\
&+ \ H^{i}_{t-} \int_{\mathbf{R}^n} (v^i(t , Y_{t-}  + F^i_t(x)
) - v^i(t , Y_{t-})) \widetilde{\Pi} (dx,dt)
\\
&+ \ H^{i}_{t-} \sum_{j:j \neq i } (v^i(t , Y_{t-}  +
\rho^{i,j}_t) - v^i(t , Y_{t-})) d M^{i,j}_t.
\end{align*}
Hence and by \eqref{eq:I2+I3} \allowdisplaybreaks
\begin{align}\allowdisplaybreaks
\nonumber
 d v(t,Y_t, C_t)
&= \sum_{i} H^i_{t-} \Bigg( \partial_t v^i(t, Y_t ) +
 \nabla v^i(t, Y_{t-}) \mu^i_t
 +
\frac{1}{2} \Tr \left( \sigma^i_t  (\sigma^i_t)^\mtop \nabla^2
v^i(t, Y_{t-}) \right)
\\ \nonumber &
+ \int_{\mathbf{R}^n}\left( v^i(t , Y_{t-} \!+\!F^i_t(x) ) -
v^i(t , Y_{t-}) - \nabla v^i(t, Y_{t-})F^i_t(x) \I_{\set{
\norm{x}  < a }} \right) \nu(dx)
\\ \label{eq:dv(t,y,c)}
& + \sum_{j:j \neq i } \left( v^j(t , Y_{t-}  +
\rho^{i,j}_t) - v^i(t , Y_{t-})
\right)\lambda^{i,j}_t \Bigg) dt
\\ \nonumber
&+ \sum_{i} H^i_{t- } \left( \nabla v^i(t, Y_{t-}) \right)^\mtop
\!\!
\sigma^i_t d W_t
\\ \nonumber &
+ \sum_{i} H^{i}_{t-} \int_{R^n} (v^i(t , Y_{t-} + F^i_t (x) ) -
v^i(t , Y_{t-})) \widetilde{\Pi} (dx,dt)
\\ \nonumber &
+
\sum_{i,j:j \neq i } (v^j(t , Y_{t-} + \rho^{i,j}_t) - v^i(t , Y_{t-})) H^{i}_{t-}d M^{i,j}_t,
\end{align}
which is precisely \eqref{eq:ito-formula}.
\end{proof}

\end{document}